\documentclass[reqno,letterpaper,11pt]{amsart}

\usepackage{hyperref}

\usepackage{palatino}

\usepackage{amsmath}
\usepackage{amsfonts,amssymb,amsmath,latexsym,wasysym,mathrsfs,bbm,stmaryrd}
\usepackage[all]{xy}
\usepackage[usenames]{color}
\usepackage{epsfig}

\usepackage{graphicx}
\usepackage{subcaption}

\usepackage{amsthm}
\usepackage{thmtools}
\usepackage{thm-restate}
\declaretheorem[numberwithin=section]{theorem}
\declaretheorem[sibling=theorem]{proposition}
\declaretheorem[sibling=theorem]{definition}
\declaretheorem[sibling=theorem]{corollary}
\declaretheorem[sibling=theorem]{lemma}

\declaretheorem[sibling=theorem]{assumption}

\declaretheorem[sibling=theorem,style=remark]{remark}

\numberwithin{equation}{section}

\usepackage[mathcal]{euscript}
\usepackage{times}

\def\R{\mathbb R}
\def\C{\mathbb C}

\def\1{\mathbbm{1}}

\def\A{\mathscr{A}}

\newcommand{\tr}{\mathrm{tr}}

\long\def\symbolfootnote[#1]#2{\begingroup%
\def\thefootnote{\fnsymbol{footnote}}\footnote[#1]{#2}\endgroup}

\usepackage[foot]{amsaddr}

\begin{document}

	\title[The Brown measure of variables with semicircular imaginary part]{The Brown measure of unbounded variables with free semicircular imaginary part}
	\author{
	Ching-Wei Ho}
	\address{Institute of Mathematics, Academia Sinica, Taipei 10617, Taiwan; Department of Mathematics, University of Notre Dame, Notre Dame,
IN 46556, United States}
	\email{cho2@nd.edu}
	\date{\today} 
	\maketitle

	\begin{abstract}		
		Let $x_0$ be an unbounded self-adjoint operator such that the Brown measure of $x_0$ exists in the sense of Haagerup and Schultz. Also let $\tilde\sigma_\alpha$ and $\sigma_\beta$ be semicircular variables with variances $\alpha\geq 0$ and $\beta>0$ respectively. Suppose $x_0$, $\sigma_\alpha$, and $\tilde\sigma_\beta$ are all freely independent. We compute the Brown measure of $x_0+\tilde\sigma_\alpha+i\sigma_\beta$, extending the recent work which assume $x_0$ is a bounded self-adjoint random variable.  We use the PDE method introduced by Driver, Hall and Kemp to compute the Brown measure. The computation of the PDE relies on a charaterization of the class of operators where the Brown measure exists. The Brown measure in this unbounded case has the same structure as in the bounded case; it has connections to the free convolution $x_0+\sigma_{\alpha+\beta}$. We also compute the example where $x_0$ is Cauchy-distributed. 
			\end{abstract}

		\tableofcontents
	
	\section{Introduction}
	An elliptic variable $c_{\alpha,\beta}$ is an element in a $W^*$-probability space $(\A,\tau)$ of the form $\tilde{\sigma}_\alpha+i\sigma_\beta$ where $\tilde\sigma_\alpha$ and $\sigma_\beta$ are free semicircular random variables with variances $\alpha> 0$ and $\beta>0$ respectively. In this paper, we also consider the case $\alpha = 0$; in this case, $c_{\alpha,\beta}$ is a ``degenerate'' elliptic variable, and is just a imaginary multiple of the semicircular variable. We do not consider the case $\alpha>0$ and $\beta=0$ in this paper. In this case, $c_{\alpha,\beta}$ is Hermitian --- a semicircular variable. The elliptic variable $c_{\alpha,\beta}$ is the limit in $\ast$-distribution of the Gaussian random matrix $\sqrt{\alpha}\tilde{X}_N+i\sqrt{\beta}X_N$ where $\tilde{X}_N$ and $X_N$ are independent Gaussian unitary ensembles (GUEs); see \cite{Voiculescu1991} and the book \cite{NicaSpeicherBook}. Also, the empirical eigenvalue distribution of the random matrix $\sqrt{\alpha}\tilde{X}_N+i\sqrt{\beta}X_N$ converges to the Brown measure of the elliptic variable $c_{\alpha,\beta}$ \cite{Girko1985} (See \cite{Brown1986,HaagerupSchultz2007} for the definition of Brown measure). The Brown measure of $c_{\alpha,\beta}$ has a simple structure; it is the uniform measure in a region bounded by a certain ellipse \cite{BianeLehner2001}.
	
	Let $x_0$ be a self-adjoint noncommutative random variable in the $W^*$-probability space $\A$ containing $c_{\alpha,\beta}$ that is freely independent from $c_{\alpha,\beta}$. In particular, $x_0$ is a \emph{bounded} random variable. The author computed the Brown measure of $x_0+c_{\alpha,\beta}$ in \cite{Ho2020}. In this paper, we extend the computation to an \emph{unbounded} self-adjoint random variable $x_0$ affiliated with $\A$ for which the Brown measure exists in the sense of Haagerup and Schultz \cite{HaagerupSchultz2007}; we refer this class of unbounded random variables as the Haagerup--Schultz class. Our results show that the Brown measure of $x_0+c_{\alpha,\beta}$ has the same structure as in the bounded case; in particuar, it has direct connections to the Brown measure of $x_0+c_{s}$, where $c_s= c_{s/2,s/2}$ is the circular variable with $s=\alpha+\beta$, and the law of $x_0+\sigma_{s}$. We also compute, as an example, the Brown measure of $x_0+c_{\alpha,\beta}$ for $\alpha\geq 0$ and $\beta>0$ when $x_0$ has a Cauchy distribution. 
	
	When $x_0$ is a bounded random variable and is the limit in distribution of a sequence $A_N$ of Hermitian random matrix independent from the GUEs $\tilde{X}_N$ and $X_N$, \'Sniady \cite{Sniady2002} shows that the empirical eigenvalue distribution of
	\[A_N+\sqrt{\alpha}\tilde{X}_N+i\sqrt{\beta}X_N,\quad \alpha>0, \beta>0\]
	converges to the Brown measure of $x_0+c_{\alpha,\beta}$, as $N\to\infty$. When $x_0$ is unbounded, computer simulations show that the Brown measure of $x_0+c_{\alpha,\beta}$ is a reasonable candidate of the limiting eigenvalue distribution of the random matrix model $A_N+\sqrt{\alpha}\tilde{X}_N+i\sqrt{\beta}X_N$, where the empirical eigenvalue distribution of $A_N$ converges to the law of $x_0$.
	
	The computation of the Brown measure in this paper is based on the PDE method introduced by Driver, Hall and Kemp in \cite{DHK2019}. In \cite{DHK2019}, Driver, Hall and Kemp computed the Brown measure of the free multiplicative Brownian motion using the Hamilton--Jacobi method. The author computes in later work with collaborators the Brown measures of other variables using the same PDE method \cite{HallHo2020, Ho2020, HoZhong2019}.
The results in these papers assume the boundedness of $x_0$ in the process of computing the PDE. The key observation in this paper is to show that the function
	\begin{equation}
	\label{eq:S}
	S(t,\lambda,\varepsilon) = \tau[\log((x_0+i\sigma_t-\lambda)^*(x_0+i\sigma_t-\lambda)+\varepsilon)], \quad t>0, \lambda\in\C, \varepsilon>0
	\end{equation}
	satisfies the same PDE derived in \cite{HallHo2020}	when $x_0$ is an unbounded random variable affiliated with $\A$ in the Haagerup--Schultz class. The solution of the PDE yields the Brown measure of $x_0+i\sigma_t$. By considering the free additive convolution $x_0+\tilde\sigma_{\alpha}$ in place of $x_0$, we derive the family of the Brown measures of $x_0+c_{\alpha,\beta}$ for $\alpha\geq0$ and $\beta>0$. In particular, by taking $\alpha =\beta>0$, we extend the results in \cite{HoZhong2019} to unbounded self-adjoint $x_0$.

	The method of deriving the PDE \eqref{eq:S} in \cite{HallHo2020} for the bounded $x_0$ case is to apply the free It\^o formula; the partial derivatives $\partial S/\partial\lambda$ and $\partial S/\partial\varepsilon$ of $S$ can also be computed easily using the formula given in \cite[Lemma 1.1]{Brown1986}. When $x_0$ is unbounded, the free It\^o formula does not apply; furthermore, since the spectrum of $(x_0+i\sigma_t-\lambda)^*(x_0+i\sigma_t-\lambda)+\varepsilon$ may be an unbounded set on the positive half-line, it takes extra work to verify that the partial derivatives of $S$ are given by the ``same'' formula as in the bounded $x_0$ case. To overcome these issues, in this paper, we use the characterization of $x_0 = AB^{-1}$ for some $A, B\in \A$ given in \cite[Lemma 2.4]{HaagerupSchultz2007} to write
	\[S(t,\lambda,\varepsilon) = \tau[\log(\vert x_0\vert^2+1)]+\tau[\log((\lambda B-C_t)^*(\lambda B-C_t)+B^2\varepsilon)]\]
	where $C_t = A+i\sigma_t B$. The point is that the above form of $S$ is written in terms of bounded operators in $\A$, except for the constant $\tau[\log(\vert x_0\vert^2+1)]$. And then we prove that the partial derivatives and the time derivative of $S$ are all well-approximated by the corresponding partial derivatives of the function
	\[T_\eta(t,\lambda,\varepsilon)=\tau[\log((\lambda B-C_t)^*(\lambda B-C_t)+B^2\varepsilon+\eta)].\]
	The extra regularization $\eta$ allows us to apply the free It\^o formula and the derivative formula \cite[Lemma 1.1]{Brown1986}.

	The paper is organized as follows. Section~\ref{sect:results} lists out all the Brown measures computed in this paper. These Brown measures are not listed in the order of being proved in this paper; they are listed according to the complexity of their structures. Section~\ref{sect:Bkground} contains some basic definitions and background from free probability theory and Brown measure. It also contains the PDE computed in \cite{HallHo2020}. Section~\ref{sect:DiffEq} contains the crucial theorem of this paper --- showing that we get the same PDE computed in \cite{HallHo2020} for the function $S$ defined in~\eqref{eq:S}, even when $x_0$ is unbounded. In Section~\ref{sect:HJ}, we analyze the PDE using the Hamilton--Jacobi method. The Hamilton's equation is solved in \cite{HallHo2020}; the solution of the Hamilton's equation is given without proof. In Section~\ref{sect:BrownSc}, we use the solution of the Hamilton's equation to analyze the logarithmic potential and compute the Brown measure of $x_0+i\sigma_t$. Since the law of $x_0$ now possibly has unbounded support, we provide a proof whenever the analysis given in \cite{HallHo2020} only works for bounded random variable $x_0$.  The Brown measure has full measure on an open set $\Omega_{x_0,t}$. In Section~\ref{sect:ellipse}, we replace the random variable $x_0$ in the results in Section~\ref{sect:BrownSc} by the free convolution $x_0+\tilde\sigma_\alpha$ to compute the Brown measure of $x_0+c_{\alpha,\beta}$. Finally, in Section~\ref{sect:example}, we compute the Brown measure of $x_0+c_{\alpha,\beta}$ when $x_0$ has a Cauchy distribution.
	
	\subsection{Statement of results\label{sect:results}}
	In this subsection, we list out the Brown measures computed in this paper according to the complexity of their structures. Let $(\A,\tau)$ be a $W^*$-probability space and $\tilde{\A}$ be the algebra of closed densely defined operators affiliated with $\A$. Let $x_0\in\tilde{\A}$ be self-adjoint satisfying the following standing assumption throughout the paper.
	\begin{assumption}
		\label{assump:standing}
		The law $\mu_{x_0}$ of $x_0$ is not a Dirac mass at a point on $\R$, and satisfies
	\begin{equation}
		\label{eq:logassump}
		\int_1^\infty\log \vert x\vert\,d\mu_{x_0}(x)<\infty.
	\end{equation}
	\end{assumption}
	The assumption \eqref{eq:logassump} means that $x_0\in\tilde{A}$ is in the Haagerup--Schultz class $\A^\Delta$; that is, the Brown measure of $x_0$ exists in the sense of Haagerup and Schultz (See Definition~\ref{def:ADeltaDef}). If $\mu_{x_0}$ is a Dirac mass, then the Brown measure of $x_0+c_{\alpha,\beta}$ is just the translation of the elliptic law in the case $\alpha>0$, or just the translation of the semicircular law on the imaginary axis in the case $\alpha= 0$. 
	
	We start by Brown measure of $x_0+c_t$, where $c_t = c_{t/2,t/2}$ is the circular variable, freely independent from $x_0$. Recall that $\sigma_t$ denotes a semicircular variable of variance $t$, freely independent from $x_0$. The following theorem is proved in Corollaries~\ref{cor:Circular} and~\ref{cor:ellipsePush}(2). When $x_0$ is bounded, this theorem is established in Theorems 3.13 and 3.14 in \cite{HoZhong2019}. 
	
	\begin{theorem}
		\label{thm:CircularIntro}
		The following statements about the Brown measure of $x_0+c_t$ hold.
		\begin{enumerate}
		\item The Brown measure of $x_0+c_t$ has full measure on an open set $\Lambda_{x_0,t}$ of the form
		\[\Lambda_{x_0,t} = \{\left.u+iv\in\C\right\vert |v|<v_{x_0,t}(u)\}\]
		for a certain function $v_{x_0,t}$ that appears in the study of the free convolution $x_0+\sigma_t$. Then the density of the Brown measure on $\Lambda_{x_0,t}$ has the special form
		\[w_t(u+iv) = \frac{1}{\pi t}\left(1-\frac{t}{2}\frac{d}{du}\int\frac{x\,d\mu_{x_0}(x)}{(u-x)^2+v_{x_0,t}(u)^2}\right)\]
		for all $u+iv\in\Lambda_{x_0,t}$; in particular, the density is strictly positive and constant along vertical segments in $\Lambda_{x_0,t}$.

		\item Consider the function
		\begin{equation}
		\label{eq:IntroPsi}
		\psi_{x_0,t}(u) = u+t\int\frac{(u-x)\,d\mu_{x_0}(x)}{(u-x)^2+v_{x_0,t}(u)^2},\quad u\in\R
		\end{equation}
		that appears in the study of the free convolution $x_0+\sigma_t$. Then the push-forward of the Brown measure of $x_0+c_t$ by the map $u+iv\mapsto \psi_{x_0,u}(u)$ is the law of $x_0+\sigma_t$.
		\end{enumerate}
		\end{theorem}
	
	The Brown measure of $x_0+c_s$ ``generates'' the family of the Brown measures of $x_0+c_{\alpha,\beta}$ with $\alpha+\beta=s$ in the sense of push-forward measure. The following theorem is established in Theorem~\ref{thm:ellipse} and Corollary~\ref{cor:ellipsePush}. This result extends Theorems 3.7 and 4.1 in \cite{Ho2020} to possibly-unbounded $x_0$ affiliated with $\A$. 
	
	\begin{theorem}
		The following about the Brown measure of $x_0+c_{\alpha,\beta}$, for $\alpha\geq0$ and $\beta>0$ holds.
		\begin{enumerate}
			\item The Brown measure of $x_0+c_{\alpha,\beta}$ has full measure on an open set $\Omega_{x_0,t}$ of the form
			\begin{equation}
			\label{eq:OmegaIntro}
			\Omega_{x_0,\alpha,\beta} = \{\left.u+iv\in\C\right\vert |v|<\varphi_{x_0,\alpha,\beta}(u)\}
			\end{equation}
			for a certain function $\varphi_{x_0,\alpha,\beta}$. There exists a strictly increasing homeomorphism $f_{x_0,\alpha,\beta}$ on $\R$ such that, by writing $u=f_{x_0,\alpha,\beta}(u_0)$, the density of the Brown measure on $\Omega_{x_0,\alpha,\beta}$ has the special form
			\[w_{\alpha,\beta}(u+iv) = \frac{1}{4\pi \beta}\left(1+2\beta\frac{d}{du}\int\frac{(u_0-x)\,d\mu_{x_0}(x)}{(u_0-x)^2+v_{x_0,s}(u_0)^2}\right),\]
			for all $u+iv\in\Omega_{x_0,\alpha,\beta}$, where $s=\alpha+\beta$.	The density is strictly positive and constant along vertical segments in $\Omega_{x_0,\alpha,\beta}$. Furthermore, by writing $u=f_{x_0,\alpha,\beta}(u_0)$ again, the function $\varphi_{x_0,\alpha,\beta}$ in~\eqref{eq:OmegaIntro} is given by $\varphi_{x_0,\alpha,\beta}(u)=v_{x_0,s}(u_0)$ where $v_{x_0,s}$ is defined as in Theorem~\ref{thm:CircularIntro}(1)
			
			\item Write $s=\alpha+\beta$.	Let $U_{\alpha,\beta}:\overline{\Lambda}_{x_0,s}\to\overline{\Omega}_{x_0,\alpha,\beta}$ be defined by
			\begin{equation*}
			U_{\alpha,\beta}(u+iv) = f_{x_0,\alpha,\beta}(u)+i\frac{2\beta v}{s}.
			\end{equation*}
			Then $U_{\alpha,\beta}$ is a homeomorphism, and the push-forward of the Brown measure of $x_0+c_s$ under the map $U_{\alpha,\beta}$ is the Brown measure of $x_0+c_{\alpha,\beta}$.
			
			\item Let $s=\alpha=\beta$. Recall the function $\psi_{x_0,s}$ is defined in~\eqref{eq:IntroPsi}. The push-forward of the Brown measure of $x_0+c_{\alpha,\beta}$ by the map
				\[Q_{\alpha,\beta}(u+iv) = \begin{cases}
				\frac{1}{\alpha-\beta}[su-2\beta f_{x_0,\alpha,\beta}^{-1}(u)]\quad &\textrm{if $\alpha\neq \beta$}\\
				\psi_{x_0,s}(u)\quad &\textrm{if $\alpha= \beta=s/2$}.
				\end{cases}\]
				is the law of $x_0+\sigma_s$.
		\end{enumerate}
		\end{theorem}
	By putting $\alpha = 0$ and $\beta = t$, the above theorem recovers Theorem~\ref{thm:a+ist} in this paper which computes the Brown measure of $x_0+i\sigma_t$. 
	
	We close this section by a matrix simulation. Let $A_N$ be an $N\times N$ diagonal random matrix whose diagonal entries are independent Cauchy-distributed random variables; that is, each diagonal entry of $A_N$ is a random variable with density
	\[\frac{1}{\pi}\frac{1}{x^2+1},\quad x\in\R.\]
	Also, let $X_N$ be an $N\times N$ GUE, independent from $A_N$. Figure~\ref{fig:Cauchy} plots a computer simulation of the eigenvalues of $A_N+iX_N$ with $N=5000$ and the density of the Brown measure of $x_0+i\sigma_1$ where $x_0$ has the Cauchy distribution. This simulation suggests that the Brown measure is a reasonable candidate of the limiting eigenvalue distribution even in the unbounded case. 
	
	\begin{figure}[h]
		\begin{center}
			\includegraphics[width=\textwidth]{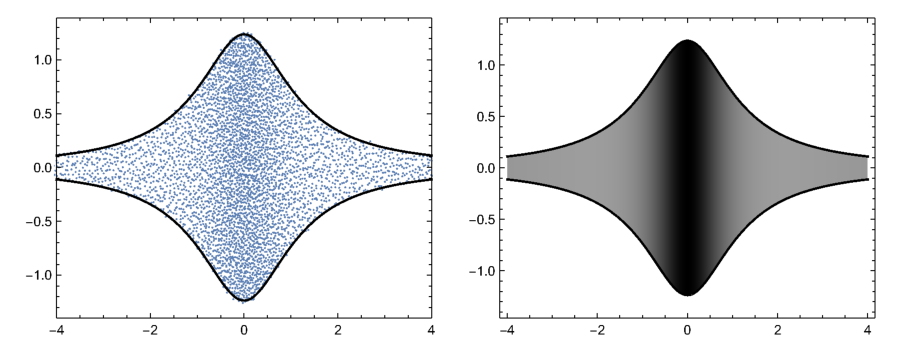}
		\end{center}
		\caption{$5000\times 5000$ matrix simulation of $A_{5000}+iX_{5000}$ and the Brown measure densty of $x_0+i\sigma_1$.\label{fig:Cauchy}}
	\end{figure}
	
	\section{Background and previous results\label{sect:Bkground}}
	\subsection{Free probability}
	We first introduce noncommutative random variables and freeness (or free independence).
	\begin{definition}
		\label{def:W*notation}
		\begin{enumerate}
			\item A {\bf $W^*$-probability space} $(\A,\tau)$ is a finite von Neumann algebra $\A$ with a normal, faithful, tracial state $\tau$. 
			\item The algebra $\tilde{\A}$ affiliated with $\A$ is the set of closed, densely defined operators such that every $a\in\A$ has a polar decomposition
			\[a = u|a| = u\int_0^\infty \lambda\,dE_{|a|}(\lambda)\]
			where $u\in\A$ is unitary and the projection-valued spectral measure $E_{|a|}$ takes values in $\A$. The elements in $\tilde{\A}$ are called (noncommutative) random variables. When we distinguish elements from $\tilde{\A}$ and $\A$, any $a\in\A$ is called a bounded random variable and any $a\in\tilde{\A}\setminus\A$ is called an unbounded random variable.
			\item Given any self-adjoint random variable $a\in\tilde{\A}$ with spectral measure $E_a$, the law $\mu_a$ of $a$ is the probability measure on $\R$ defined by
			\[\mu_a(B) = \tau(E_a(B))\]
			for any Borel set $B$. 
			\item The von Neumann subalgebras $\A_1,\ldots,\A_n\subset \A$ are said to be free, or freely independent, if given any $i_1,\ldots,i_m\in\{1,\ldots,n\}$ with $i_k\neq i_{k+1}$ and $a_{i_k}\in\A_{i_k}$ for all $k$ satisfying $\tau(a_{i_k}) = 0$ for all $1\leq k\leq m$, we have
			\[\tau(a_{i_1}\cdots a_{i_m}) = 0.\]
			The random variables $a_1,\ldots,a_m\in\tilde{\A}$ are free, or freely independent, if the von Neumann subalgebras $\A_k$ generated by all the spectral projections of $a_{k}$ are free.
			\end{enumerate}
		\end{definition}
	In this paper, we denote $\sigma_t$ or $\tilde\sigma_t$ by the semicircular variable with variance $t$; that is, the law $\mu_{\sigma_t}$ is given by
	\[\frac{1}{2\pi t}\sqrt{4t-x^2}\1_{|x|\leq 2\sqrt{t}}\,dx.\]
	
	One way to study the law of $a\in\tilde{\A}$ is to use the {\bf Cauchy transform} $G_a$ of $a$ defined by
	\[G_a(z) = \tau[(z-a)^{-1}] = \int\frac{d\mu_a(x)}{z-x},\quad z\in\C^+,\]
	where $\C^+$ is the upper half plane. The Cauchy transform is determined by the law $\mu_a$ of $a$. Given any finite measure $\mu$ on $\R$, one can define the Cauchy transform $G_\mu$ of $\mu$ by replacing $\mu_a$ in the right hand side of the above equation by $\mu$. We can recover the law $\mu_a$ from $G_a$ by the Stieltjes inversion formula 
	\[d\mu_a(x)=-\frac{1}{\pi}\lim_{\varepsilon\to 0^+}\mathrm{Im}\,G_a(x+i\varepsilon)\,dx\]
	where the limit is in weak sense.
	
	The Cauchy transform $G_a$ of $a$ is invertible on some truncated Stolz angle
	\[\Gamma_{\alpha,\beta}=\{\left.x+iy\in\C\right\vert y>\beta \textrm{ and } x<\alpha y\}\]
	for positive $\alpha$ and $\beta$. The inverse of $G_a$ is called the {\bf $K$-transform} of $a$ and is denoted by $K_a$. The {\bf $R$-transform} $R_a$ is defined as
	\[R_a(z) = K_a(z) - \frac{1}{z}.\]
	The following fact about $R$-transform was first discovered by Voiculescu \cite{Voiculescu1986} for bounded random variables, then by Maassen \cite{Maassen1992} for possibly unbounded random variables with finite variance, and by Bercovici and Voiculescu \cite{BercoviciVoiculescu1993} for arbitrary unbounded random variables.
	\begin{theorem}
		If $a,b\in\tilde{\A}$ are self-adjoint freely independent random variables, then, on the common domain where the $R$-transforms $R_a$, $R_b$, and $R_{a+b}$ are defined,
		\[R_{a+b} = R_a+R_b.\]
		\end{theorem}

	\subsection{Free addition convolution}
	When $x_0\in\A$ is a bounded self-adjoint random variable free from $\sigma_t$, the Brown measure of $x_0+i\sigma_t$ has direct connections to $x_0+\sigma_t$, by \cite{HallHo2020, HoZhong2019}. In this paper, we will show that this is also the case for unbounded self-adjoint random variable $x_0$. In this section, we review the free additive convolution of $x_0$ and $\sigma_t$, where $x_0$ can be taken to be unbounded.
	
	Biane \cite{Biane1997sc} computed the free additive convolution of a self-adjoint random variable $x_0\in\tilde{\A}$ and the semicircular variable $\sigma_t$. Before we state the results, we first introduce some notations.
	\begin{definition}
		\label{def:SCConv}
		Denote by $\mu_{x_0}$ the law of the self-adjoint random variable $x_0\in\tilde{\A}$.
		\begin{enumerate}
			\item Let $H_{x_0,t}$ be defined by
			\begin{equation}
				\label{eq:Htdef}
				H_{x_0,t}(z) = z + t\int\frac{d\mu_{x_0}(x)}{z-x},\quad z\not\in\mathrm{supp}(\mu_{x_0})
				\end{equation}
			where $\mathrm{supp}(\mu_{x_0})$ denotes the support of $\mu_{x_0}$.
			\item Define a nonnegative function on $\R$ by
			\[v_{x_0,t}(u) = \inf\left\{v > 0\left\vert \int\frac{d\mu_{x_0}(x)}{(u-x)^2+v^2}\leq\frac{1}{t}\right.\right\}.\]
			This means, for a given $u\in\R$, $v_{x_0,t}(u)$ is the unique positive number (if exists) such that $\int\frac{d\mu_{x_0}(x)}{(u-x)^2+v_{x_0,t}(u)^2}=\frac{1}{t}$; if such a number does not exist, $v_{x_0,t}(u) = 0$.
			\item Let $\Gamma_{x_0,t} = \{u+iv\in\C\left\vert v>v_{x_0,t}(u)\right.\}$ be the region in the upper half plane $\C^+$ that is above the graph of $v_{x_0,t}$.
			\end{enumerate}
		\end{definition}
	
	Recall that $G_a$ denotes the Cauchy transform of $a\in\tilde{\A}$. The following theorem, due to Biane \cite{Biane1997sc}, computes the free additive convolution of a self-adjoint random variable $x\in\tilde{\A}$ and the semicircular variable $\sigma_t$.
	\begin{theorem}[\cite{Biane1997sc}]
		\label{thm:Biane}
		\begin{enumerate}
			\item The function $H_{x_0,t}$ is an injective conformal map from $\Gamma_{x_0,t}$ onto the upper half plane $\C^+$. The function extends to a homeomorphism from $\overline{\Gamma}_{x_0,t}$ onto $\C^+\cup\R$. Hence, for any $u\in\R$, $H_{x_0,t}(u+iv_{x_0,t}(u))$ is real.
			\item The function $H_{x_0,t}$ satisfies
			\begin{equation}
				\label{eq:scsubord}
				G_{x_0+\sigma_t}(H_{x_0,t}(z)) = G_{x_0}(z).
				\end{equation}
			\item The law of $x_0+\sigma_t$ is absolutely continuous with respect to the Lebesgue measure; its density $p_{x_0,t}$ can be computed using the function $\psi_{x_0,t}(u) = H_{x_0,t}(u+iv_{x_0,t}(u))$. The function $\psi_{x_0,t}:\R\to\R$ is a homeomorphism, and
			\[p_{x_0,t}(\psi_{x_0,t}(u)) = \frac{v_{x_0,t}(u)}{\pi t}.\]
			\end{enumerate}
		\end{theorem}

	When $a$ and $b$ are arbitrary self-adjoint random variables, the free additive convolution $a+b$ can be studied using Cauchy transforms in a way similar to \eqref{eq:scsubord} (the method of subordination functions). The subordination functions are first discovered by Voiculescu \cite{Voiculescu1993} for bounded random variables $a$ and $b$. Then it is further developed by Biane \cite{Biane1998}, and again by Voiculescu \cite{Voiculescu2000} to a even more general setting. Belinschi and Bercovici \cite{BelinschiBercovici2007} use a complex analysis approach to show the existence of subordination functions (where they do not assume the random variables to be bounded).
	
	\subsection{The Brown measure}
	\label{sect:BrownDef}
	In this section, we review the definition of the Brown measure, first introduced by Brown \cite{Brown1986} for operators in a von Neumann algebra with a faithful, semifinite, normal trace, then extended to a class $\A^\Delta$ of unbounded operators by Haagerup and Schultz \cite{HaagerupSchultz2007}. We refer this class of unbounded operators as the Haagerup--Schultz class. In the following, we introduce the class $\A^\Delta$. 
	\begin{definition}
		\label{def:ADeltaDef}
		We denote by $\A^\Delta$ the set of operators $a$ in $\tilde{\A}$ (see Definition~\ref{def:W*notation}(2)) such that
		\[\tau(\log^+|a|) = \int_0^\infty \log^+(\lambda)\,d\mu_{|a|}(\lambda)\]
		where $\log^+$ is defined to be $\log^+(x) = \max(\log x, 0)$ for $x\geq 0$. 
		\end{definition}
	
	Let $a\in\A^{\Delta}$. Define a function $S$ by
	\[S(\lambda,\varepsilon) = \tau[\log(|a-\lambda|^2+\varepsilon)], \lambda\in\C, \varepsilon>0.\] 
	Then
	\[s(\lambda) = \lim_{\varepsilon\to 0^+} S(\lambda,\varepsilon), \quad \lambda\in\C\]
	is a subharmonic function, with value in $\R\cup\{-\infty\}$. The exponential of $s(\lambda)/2$ is the Fuglede-Kadison determinant of $a-\lambda$. The Fuglede-Kadison determinant for operators in a finite factor was introduced by Fuglede and Kadison \cite{FugledeKadison1952}. The {\bf Brown measure} of $a$ is defined to be
	\begin{equation*}
		\mathrm{Brown}(a) = \frac{1}{4\pi}\Delta s(\lambda)
		\end{equation*}
	where the Laplacian is in distributional sense.
	
	The first statement of the following theorem characterizes the operators in $\A^\Delta$, and is proved in Lemma 2.4 of \cite{HaagerupSchultz2007}. The second statement is an important property of $\A^\Delta$, and is proved in Propositions 2.5 and 2.6 of \cite{HaagerupSchultz2007}. 
		\begin{theorem}[\cite{HaagerupSchultz2007}]
			\label{thm:ADelta}
			\begin{enumerate}
		\item Let $a\in\tilde{\A}$. Then $a\in\A^\Delta$ is equivalent to any of the following.
		\begin{enumerate}
			\item There exist $x_1\in\A$ and $y_1\in\A$ such that $a = x_1y_1^{-1}$ and
		\[\tau[\log|a|] = \tau[\log|x_1|]-\tau[\log|y_1|].\]
			\item There exist $x_2\in\A$ and $y_2\in\A$ such that $a = x_2^{-1}y_2$ and
			\[\tau[\log|a|] = \tau[\log|y_2|]-\tau[\log|x_2|]\]
	\end{enumerate}
		\item The set $\A^\Delta$ is an algebra; that is, it is a subspace in $\tilde{\A}$ and is closed under multiplication. 
		\end{enumerate}
	\end{theorem}
	\begin{remark}
		\label{rem:Brownelig}
		An immediate consequence from Theorem~\ref{thm:ADelta} is that if $x_0\in\A^\Delta$, then $x_0+c_{\alpha,\beta}\in\A^\Delta$, for all $\alpha\geq 0$ and $\beta>0$, so that the Brown measure of $x_0+c_{\alpha,\beta}$ is defined in the sense of Haagerup and Schultz.
		\end{remark}

		The following proposition is \cite[Proposition 2.5]{HaagerupSchultz2007}.
		\begin{proposition}[\cite{HaagerupSchultz2007}]
			\label{prop:FKMult}
			Given any $x, y\in\A^\Delta$, 
			\[\tau[\log \vert xy\vert] = \tau[\log\vert x\vert]+\tau[\log\vert y\vert].\]
		\end{proposition}

	\subsubsection{A PDE for the Brown measure of $x_0+i\sigma_t$ with bounded $x_0$}
	In \cite{HallHo2020}, Hall and the author computed the Brown measure of $x_t = x_0+i\sigma_t$ where $x_0\in\A$ is a (bounded) self-adjoint random variable with law $\mu_{x_0}$, and $\sigma_t\in\A$ is a semicircular variable, freely independent from $x_0$. Define
	\[S(t, \lambda,\varepsilon) = \tau[\log((x_t-\lambda)^*(x_t-\lambda)+\varepsilon)], \quad t\geq 0, \lambda\in\C, \varepsilon>0.\]
	Then the function $S$ satisfies the following PDE of Hamilton--Jacobi type.
	\begin{theorem}
		\label{thm:PDE}
		Write $\lambda = u+iv$. Then $S$ satisfies the PDE
		\[\frac{\partial S}{\partial t} = \varepsilon \left(\frac{\partial S}{\partial \varepsilon}\right)^2+\frac{1}{4}\left(\left(\frac{\partial S}{\partial u}\right)^2-\left(\frac{\partial S}{\partial v}\right)^2\right)\]
		with initial condition
		\[S(0,\lambda,\varepsilon) = \int\log((x-\lambda)^*(x-\lambda)+\varepsilon)\,d\mu_{x_0}(x).\]
		\end{theorem}
	Hall and the author \cite{HallHo2020} analyze the function $S$ using the above PDE to compute the Brown measure of $x_t$. The key observation in this paper is to extend Theorem~\ref{thm:PDE} to self-adjoint unbounded $x_0\in\A^\Delta$. The solution of this PDE yields the Brown measure of $x_0+i\sigma_t$. The proof from some of the other results in \cite{HallHo2020} will also need to modified for $\mu_{x_0}$ with unbounded support on $\R$. We do not list all the results of \cite{HallHo2020} here; rather, we will state the results in the process of computing the Brown measure, with modified proof whenever necessary.

\section{The PDE for $S$\label{sect:DiffEq}}
\subsection{Outline of the unbounded case}
Let $x_0\in \A^\Delta$ (See Definition~\ref{def:ADeltaDef}) be self-adjoint and write $x_t = x_0+i\sigma_t$ where $\sigma_t$ is a semicircular variable, freely independent from $x_0$. By Remark~\ref{rem:Brownelig}, the Brown measure of $x_0+i\sigma_t$ exists in the sense of Haagerup and Schultz \cite{HaagerupSchultz2007}. 

Define
\[S(t,\lambda,\varepsilon)=\tau[\log((x_t-\lambda)^*(x_t-\lambda)+\varepsilon)].\]
The PDE in Theorem~\ref{thm:PDE} was proved in \cite{HallHo2020} under the assumption that $x_0$ is a bounded self-adjoint random variable in $\A$. In Section~\ref{sect:PDEproof}, we will show that $S$ satisfies the same PDE as in the bounded $x_0$ case (Theorem~\ref{thm:PDE}).
\begin{theorem}
	\label{thm:UnboundedPDE}
	Write $\lambda = u+iv$. Then $S$ satisfies the PDE
	\[\frac{\partial S}{\partial t} = \varepsilon \left(\frac{\partial S}{\partial \varepsilon}\right)^2+\frac{1}{4}\left(\left(\frac{\partial S}{\partial u}\right)^2-\left(\frac{\partial S}{\partial v}\right)^2\right)\]
	with initial condition
	\[S(0,\lambda,\varepsilon) = \int\log((x-\lambda)^*(x-\lambda)+\varepsilon)\,d\mu_{x_0}(x).\]
\end{theorem}

By definition of $\tilde{A}$, $x_0$ has the polar decomposition $u\vert x_0\vert$. Let
\begin{align}
	A&=u\left\vert x_0\right\vert(\left\vert x_0\right\vert+1)^{-1/2}\nonumber\\
	B &= (\left\vert x_0\right\vert^2+1)^{-1/2}\nonumber\\
	C_t &= A+i\sigma_tB \label{eq:AB}
\end{align}
be bounded variables in $\A$. We can then represent $x_0 = AB^{-1}$ as a ``fraction" of two variables in $\A$. This representation is discovered by Haagerup and Schultz \cite{HaagerupSchultz2007} and is essential in the analysis of the Brown measure of unbounded variables. The notation $C_t$ is just for convenient.

\begin{proposition}
	Suppose $A$ and $B$ are as in \eqref{eq:AB}. Then we have
	\[S(t,\lambda,\varepsilon) = \tau[\log(\vert x_0\vert^2+1)]+\tau[\log((\lambda B - C_t)^*(\lambda B-C_t)+B^2\varepsilon)].\]
\end{proposition}
We note that $\vert x_0\vert^2+1\in\A^\Delta$ by Theorem~\ref{thm:ADelta}; by the definition of $\A^\Delta$, it is clear that $0\leq \tau[\log(\vert x_0\vert^2+1)]<\infty$.
\begin{proof}
	The operator $B$ is self-adjoint. Using $x_0 = AB^{-1}$, we have
	\begin{align*}
		S(t,\lambda,\varepsilon) &= \tau[\log(B^{-1}((\lambda B - C_t)^*(\lambda B - C_t)+B^2\varepsilon)B^{-1})]\\
		&= \tau[\log(B^{-2})]+\tau[\log((\lambda B - C_t)^*(\lambda B-C_t)+B^2\varepsilon)]
	\end{align*}
	by Proposition~\ref{prop:FKMult}. The conclusion follows from the definition of $B$.
\end{proof}

We attempt to apply the free It\^o formula to $\log((\lambda B - C_t)^*(\lambda B-C_t)+B^2\varepsilon)$ and compute the PDE that it satisfies. In the bounded $x_0$ case considered in \cite{HallHo2020}, we apply \cite[Lemma 1.1]{Brown1986} to compute the $\lambda$- and $\varepsilon$- partial derivatives of
\[\tau[\log((\lambda-x_0-i\sigma_t)^*(\lambda-x_0-i\sigma_t)+\varepsilon)].\]
However, when $x_0$ is unbounded, $(\lambda B - C_t)^*(\lambda B-C_t)+B^2\varepsilon$ is not necessarily positive (that is, $0$ is possibly in the spectrum). We cannot apply \cite[Lemma 1.1]{Brown1986} to compute the $\lambda$- and $\varepsilon$- partial derivatives of
\[\tau[\log((\lambda B - C_t)^*(\lambda B-C_t)+B^2\varepsilon)].\]
We also have a similar issue when we apply the free It\^o formula to compute the $t$-derivative. The strategy is to consider another regularization as follows.
\begin{definition}
	\label{def:Teta}
	For any $\eta>0$, we define
	\[T_\eta(t,\lambda,\varepsilon) = \tau[\log((\lambda B - C_t)^*(\lambda B-C_t)+B^2\varepsilon+\eta)].\]
\end{definition}
We then prove that the derivatives of $S$ are all well-approximated by $T_\eta$, and that $S$ satisfies the same PDE as $T_\eta$. Since $(\lambda B - C_t)^*(\lambda B-C_t)+B^2\varepsilon$ is nonnegative, it is easy to see that 
\begin{equation}
	\label{eq:Sptwise}
	T_\eta(t,\lambda,\varepsilon)\to S(t,\lambda,\varepsilon) - \tau[\log(\vert x_0\vert^2+1)]
\end{equation}
pointwise. In Section~\ref{sect:lambdaepsilonDer}, we will compute the partial derivatives of $S$ with respect to $\lambda$ and $\varepsilon$, by taking the limit of the corresponding derivatives of $T_\eta$ as $\eta\downarrow 0$. In Section~\ref{sect:PDEproof}, we will apply the free It\^o formula to compute the PDE of $T_\eta$, and then take limit as $\eta\downarrow 0$ to compute the PDE of $S$. For convenience, we employ the notations
\[R_\eta = (\lambda B - C_t)^*(\lambda B-C_t)+B^2\varepsilon+\eta)^{-1}\]
and
\[R = ((\lambda B - C_t)^*(\lambda B-C_t)+B^2\varepsilon)^{-1}.\]

\subsection{The $\lambda$- and $\varepsilon$-derivatives}
\label{sect:lambdaepsilonDer}
\begin{proposition}
	\label{prop:partials}
	Write $\tilde{R} = (\vert\lambda-x_0-i\sigma_t\vert^2+\varepsilon)^{-1}$. The $\lambda$-, $\bar\lambda$-, and $\varepsilon$-derivatives of $S$ are
	\begin{align*}
		\frac{\partial S}{\partial\lambda} &= \tau[(\lambda-x_0-i\sigma_t)^*\tilde{R})]\\
		\frac{\partial S}{\partial\bar\lambda} &= \tau[(\lambda-x_0-i\sigma_t)\tilde{R}]\\
		\frac{\partial S}{\partial\varepsilon} &= \tau[\tilde{R}].
	\end{align*}
	These traces are all finite numbers. Furthermore, they are the pointwise limit of the corresponding partial derivatives of $T_\eta$ as $\eta\downarrow 0$.
\end{proposition}
The proof uses the fundamental theorem of calculus. The idea is that the partial derivatives of $S$ can be approximated by the corresponding partial derivatives of $T_\eta$. The partial derivatives of $T_\eta$ are recorded as follows.
\begin{lemma}
	\label{lem:Tpartials}
	The $\lambda$-, $\bar\lambda$-, and $\varepsilon$-derivatives of $T_\eta$ are
	\begin{align*}
		\frac{\partial T_\eta}{\partial\lambda} &= \tau[R_\eta (\lambda B-C_t)^*B]\\
		\frac{\partial T_\eta}{\partial\bar\lambda} &= \tau[R_\eta B^*(\lambda B-C_t)]\\
		\frac{\partial T_\eta}{\partial\varepsilon} &= \tau[R_\eta B^2].
	\end{align*}
\end{lemma}
\begin{proof}
	In the defintion of $T_\eta$, the argument of $\log$ is a positive bounded operator. These partial derivatives can be computed using \cite[Lemma 1.1]{Brown1986}.
\end{proof}
The following lemma is a direct consequence of that the trace is a state.
\begin{lemma}
	\label{lem:finitetrace}
	Given any operator $X$ affiliated to $\A$, we have 
	\[\tau[\vert X\vert^2(\vert X\vert^2+\varepsilon)^{-2}]\leq \frac{1}{4\varepsilon}\]
	and
	\[\tau[\vert X\vert^4(\vert X\vert^2+\varepsilon)^{-2}]\leq 1\]
\end{lemma}
\begin{proof}
	This follows directly from the facts that $\|\vert X\vert^2(\vert X\vert^2+\varepsilon)^{-2}\|\leq 1/4\varepsilon$ and $\|\vert X\vert^4(\vert X\vert^2+\varepsilon)^{-2}\|\leq 1$.
\end{proof}
\begin{lemma}
	\label{lem:lambdaCont}
	Let $X$ be an operator affiliated with $\A$ and $Y\in \A$. Given any $\varepsilon>0$, the functions
	\[f(\lambda) = \tau[(\lambda Y-X)(\vert\lambda Y-X\vert^2+\varepsilon)^{-1}]\]
	and
	\[g(\lambda) = \tau[(\vert\lambda Y-X\vert^2+\varepsilon)^{-1}]\]
	are continuous for all $\lambda\in\C$.
\end{lemma}
\begin{proof}
	We assume, without loss of generality, $\|Y\|=1$. Let $\lambda, z\in\C$. We compute
	\begin{align*}
		&(\lambda Y-X)(\vert\lambda Y-X\vert^2+\varepsilon)^{-1} - (z Y-X)(\vert z Y-X\vert^2+\varepsilon)^{-1}\\
		&= (\lambda-z)Y(\vert\lambda Y-X\vert^2+\varepsilon)^{-1} \\
		&\qquad+(z Y-X)(\vert z Y-X\vert^2+\varepsilon)^{-1}(\vert z Y-X\vert^2-\vert \lambda Y-X\vert^2)(\vert \lambda Y-X\vert^2+\varepsilon)^{-1}\\
		&= (\lambda-z)Y(\vert\lambda Y-X\vert^2+\varepsilon)^{-1} \\
		&\qquad+(\vert z\vert^2-\vert\lambda\vert^2)(zY-X)(\vert z Y-X\vert^2+\varepsilon)^{-1}Y^*Y(\vert \lambda Y-X\vert^2+\varepsilon)^{-1}\\
		&\qquad+\overline{(\lambda-z)}(zY-X)(\vert z Y-X\vert^2+\varepsilon)^{-1}Y^*X(\vert \lambda Y-X\vert^2+\varepsilon)^{-1}\\
		&\qquad+(\lambda -z)(z Y-X)(\vert z Y-X\vert^2+\varepsilon)^{-1}X^*Y(\vert \lambda Y-X\vert^2+\varepsilon)^{-1}.
	\end{align*}
	and analyze the above expression term-by-term. 
	
	For the first term, the norm of $(\lambda-z)Y(\vert\lambda Y-X\vert^2+\varepsilon)^{-1}$ is bounded above by $\vert\lambda-z\vert/\varepsilon$ (since we assume $\|Y\| =1$), so
	\[\left\vert\tau[(\lambda -z)Y(\vert\lambda Y-X\vert^2+\varepsilon)^{-1}]\right\vert\leq \vert\lambda-z\vert/\varepsilon.\] 
	For the second term, we apply Cauchy--Schwarz inequality and Lemma~\ref{lem:finitetrace} and get
	\begin{align*}
		&\left\vert\tau[(\vert z\vert^2-\vert\lambda\vert^2)(zY-X)(\vert zY-X\vert^2+\varepsilon)^{-1}Y^*Y(\vert \lambda Y-X\vert^2+\varepsilon)^{-1}]\right\vert^2\\
		&\leq (\vert z\vert^2-\vert\lambda\vert^2)^2\tau[\vert zY-X\vert^2(\vert zY-X\vert^2+\varepsilon)^{-2}]\tau[\vert Y\vert^4(\vert\lambda Y-X\vert^2+\varepsilon)^{-2}]\\
		&\leq (\vert z\vert^2-\vert\lambda\vert^2)^2\frac{1}{4\varepsilon^3}.
	\end{align*}
	The third term can be written as
	\begin{align*}
		&\overline{(\lambda-z)}(zY-X)(\vert zY-X\vert^2+\varepsilon)^{-1}Y^*(X-\lambda Y)(\vert\lambda Y-X\vert^2+\varepsilon)^{-1}\\
		&+ \lambda\overline{(\lambda-z)}(zY-X)(\vert zY-X\vert^2+\varepsilon)^{-1}Y^* Y(\vert\lambda Y-X\vert^2+\varepsilon)^{-1};
	\end{align*}
	thus, using Cauchy--Schwarz inequality and Lemma~\ref{lem:finitetrace} again,
	\begin{align*}
		&\quad\left\vert\tau[\overline{(\lambda-z)}(zY-X)(\vert z Y-X\vert^2+\varepsilon)^{-1}Y^*X(\vert \lambda Y-X\vert^2+\varepsilon)^{-1}]\right\vert\\
		&\leq \left\vert\lambda-z\right\vert\left(\frac{1}{4\varepsilon}+\vert\lambda\vert\frac{1}{2\varepsilon^{3/2}}\right).
	\end{align*}
	Finally, we write the fourth term as
	\begin{align*}
		&(\lambda-z)(zY-X)(\vert zY-X\vert^2+\varepsilon)^{-1}(X-z Y)^*Y(\vert\lambda Y-X\vert^2+\varepsilon)^{-1}\\
		&+ \bar z(\lambda-z)(zY-X)(\vert zY-X\vert^2+\varepsilon)^{-1}Y^* Y(\vert\lambda Y-X\vert^2+\varepsilon)^{-1},
	\end{align*}
	which gives, by Cauchy--Schwarz inequality,
	\begin{align*}
		&\quad\left\vert\tau[(\lambda-z)(zY-X)(\vert z Y-X\vert^2+\varepsilon)^{-1}X^*Y(\vert \lambda Y-X\vert^2+\varepsilon)^{-1}]\right\vert\\
		&\leq \left\vert\lambda-z\right\vert \sqrt{\tau[(\vert \lambda Y-X\vert^2+\varepsilon)^{-2}]}\;\times\\
		&\:\left(\sqrt{\tau[\vert zY-X\vert^4(\vert zY-X\vert^2+\varepsilon)^{-2}]}+\vert z\vert \sqrt{\tau[\vert zY-X\vert^2(\vert zY-X\vert^2+\varepsilon)^{-2}]}\right)\\
		&\leq \vert\lambda-z\vert \frac{1}{\varepsilon}\left(1+\frac{\vert z\vert}{4\varepsilon}\right).
	\end{align*}
	The last inequality follows from Lemma~\ref{lem:finitetrace}. Therefore, we have
	\[\vert f(\lambda)-f(z)\vert\leq \vert\lambda-z\vert\left(\frac{1}{\varepsilon}+\frac{\sqrt{\varepsilon}+2\vert\lambda\vert}{4\varepsilon^{3/2}}+\frac{4\varepsilon+\vert z\vert}{4\varepsilon^2}\right)+\frac{\left\vert \vert z\vert^2-\vert\lambda\vert^2\right\vert}{2\varepsilon^{3/2}}.\]
	
	The above estimate shows that $f$ is continuous. A similar argument shows that $g$ is also continuous.
\end{proof}
\begin{proof}[Proof of Proposition~\ref{prop:partials}]
	Fix any $u_0\in \R$. By Lemma~\ref{lem:Tpartials} and the fundamental theorem of calculus,
	\begin{equation}
		\label{eq:fundCalc}
		T_\eta(t,u+iv,\varepsilon)=T_\eta(t,u_0+iv,\varepsilon) + \int_{u_0}^u \mathrm{Re}\left[\frac{\partial T_\eta}{\partial\lambda}(t,y+iv,\varepsilon)\right]\,dy.
	\end{equation}
	First, since $R (\lambda B-C_t)^*B = (\lambda-x_0-i\sigma_t)(\vert\lambda-x_0-i\sigma_t\vert^2+\varepsilon)^{-1}$, by Lemma~\ref{lem:finitetrace}, $\tau[R (\lambda B-C_t)^*B]$ is a finite number. We compute
	\begin{equation}
		\label{eq:DerDiff}
		\frac{\partial T_\eta}{\partial\lambda}(t,y+ib,\varepsilon)- \tau[R(\lambda B-C_t)^*B] = -\eta R_\eta R (\lambda B-C_t)^*B.
	\end{equation}
	By Cauchy--Schwarz inequality,
	\begin{equation}
		\label{eq:CSineq}
		\left\vert\tau[\eta R_\eta R (\lambda B-C_t)^*B]\right\vert\leq \tau[\eta^2 R_\eta^2]\tau[\vert\lambda-x_0-i\sigma_t\vert^2(\vert\lambda-x_0-i\sigma_t\vert^2+\varepsilon)^{-2}].
	\end{equation}
	Observe that $\tau[\eta^2 R_\eta^2]\to 0$ as $\eta\downarrow 0$ by the dominated convergence theorem. By~\eqref{eq:CSineq}, $\left\vert\tau[\eta R_\eta R (\lambda B-C_t)^*B]\right\vert$ is bounded above uniformly by $1/4\varepsilon$ for all $t, \lambda$ with $\varepsilon$ fixed (see Lemma~\ref{lem:finitetrace}). By the dominated convergence theorem,
	\begin{equation}
		\label{eq:etalimit}
		\int_{u_0}^u \frac{\partial T_\eta}{\partial\lambda}(t,y+iv,\varepsilon)\,dy \to\int_{u_0}^u \tau[(\lambda-x_0-i\sigma_t)(\vert\lambda-x_0-i\sigma_t\vert^2+\varepsilon)^{-1}]\,dy \end{equation}
	as $\eta\downarrow 0$ where $\lambda = y+iv$. Recall that $T_\eta\to S - \log(\vert x_0\vert^2+1)$ pointwise. By Lemma~\ref{lem:lambdaCont}, after taking $\eta\to0$ in \eqref{eq:fundCalc}, we conclude by the fundamental theorem of calculus that
	\[\frac{\partial S}{\partial a}(t, u+iv,\varepsilon) = \tau[(\lambda-x_0-i\sigma_t)(\vert\lambda-x_0-i\sigma_t\vert^2+\varepsilon)^{-1}].\]
	We can similarly compute $\partial S/\partial b$ using the imaginary part of \eqref{eq:etalimit}. It then follows that $\partial S/\partial\lambda$ and $\partial S/\partial\bar\lambda$ are as claimed. The quantity $\partial S/\partial \varepsilon$ can be computed in a similar fashion.

	That the partial derivative $\partial S/\partial\lambda$ is the pointwise limit of $\partial T_\eta/\partial\lambda$ follows from \eqref{eq:DerDiff} and \eqref{eq:CSineq} with an application of the dominated convergence theorem. The other two derivatives $\partial S/\partial\bar\lambda$ and $\partial S/\partial\varepsilon$ follow from a similar argument.
\end{proof}

\subsection{The PDE of $T_\eta$}

We use the free It\^o formula to compute a PDE for $T_\eta$. The free stochastic integration is developed by, for example, \cite{BianeSpeicher1998, KummererSpeicher1992}. The It\^o formula that we apply in this paper is formulated in \cite{{Nikitopoulos2021}}. If $x_r$ and $\tilde{x}_r$ are two freely independent semicircular Brownian motions, they satisfy free It\^o rules given by
	\begin{align}
		\tr[A_r\, dx_r] &= \tr[A_r\,d\tilde{x}_r] = 0\nonumber\\
		dx_r\,A_r\,dx_r &=d\tilde{x}_r\,A_r\,d\tilde{x}_r = \tr[A_r]\,dr\nonumber\\
		dx_r\,A_r\,d\tilde{x}_r &=d\tilde{x}_r\,A_r\,dx_r = 0,\label{eq:scIto}
	\end{align}
	together with a stochastic product rule for processes $m_r^1$ and $m_r^2$ satisfying SDEs involving $x_r$ and $\tilde{x}_r$ (\cite[Theorem 3.2.5]{Nikitopoulos2021})
	\begin{equation}
		\label{eq;Itoprod}
		d(m^1m^2)_r = dm_r^1\,m_r^2+m_1^1\,dm_r^2+(dm_r^1)(dm_r^2).
	\end{equation}

	The following proposition is proved in \cite[Example 3.5.5]{Nikitopoulos2021}  with a circular Brownian motion. It is reformulated in two freely independent semicircular Brownian motions in \cite{HallHo2021}.
\begin{proposition}[\cite{Nikitopoulos2021}]
	\label{ito.prop}Suppose $x_{t}$ and $\tilde{x}_{t}$ are two freely independent
	semicircular Brownian motions in $\mathcal{A}$ and suppose $m_{t}$ is a
	self-adjoint process in $\mathcal{A}$ satisfying a free SDE of the form
	\[
	dm_{t}=\sum_{j=1}^{n}(a_{t}^{j}~dx_{t}~b_{t}^{j}+c_{t}^{j}~d\tilde{x}%
	_{t}~d_{t}^{j})+e_{t}~dt
	\]
	for some continuous, adapted processes $a_{t}^{j},$ $b_{t}^{j},$ $c_{t}^{j},$
	$d_{t}^{j},$ and $e_{t}.$ Then we have%
	\begin{align*}
	\frac{d}{dt}\mathrm{tr}\left[  \log(m_{t}+\eta)\right]   &
	=\mathrm{tr}\left[  (m_{t}+\eta)^{-1}e_{t}\right]  \\
	&  -\frac{1}{2}\sum_{j,k=1}^{n}\mathrm{tr}\left[  (m_{t}+\eta%
	)^{-1}a_{t}^{j}b_{t}^{k}\right]  \mathrm{tr}\left[  (m_{t}+\varepsilon
	^{2})^{-1}a_{t}^{k}b_{t}^{j}\right]  \\
	&  -\frac{1}{2}\sum_{j,k=1}^{n}\mathrm{tr}\left[  (m_{t}+\eta%
	)^{-1}c_{t}^{j}d_{t}^{k}\right]  \mathrm{tr}\left[  (m_{t}+\eta)^{-1}c_{t}^{k}d_{t}^{j}\right]  ,
	\end{align*}
	which can be written, using the free It\^{o} rules (\ref{eq:scIto}), as
	\begin{align*}
	d\mathrm{tr}\left[  \log(m_{t}+\eta)\right]  &=\mathrm{tr}\left[
	(m_{t}+\eta)^{-1}dm_{t}\right] \\
	& \qquad-\frac{1}{2}\mathrm{tr}\left[
	(m_{t}+\eta)^{-1}dm_{t}(m_{t}+\eta)^{-1}dm_{t}\right]  .
	\end{align*}
	\end{proposition}
	The definition of $T_\eta$ (Definition~\ref{def:Teta}) suggests us to apply Proposition~\ref{ito.prop} with 
	\[m_t = (\lambda B - C_t)^*(\lambda B-C_t)+B^2\varepsilon.\]
	Recall that we introduce the notation $R_\eta = (\lambda B - C_t)^*(\lambda B-C_t)+B^2\varepsilon+\eta)^{-1}$.
	\begin{proposition}
		\label{prop:TetatDer}
		The $t$-derivative of $T_\eta$ is given by 
		\begin{equation}
			\label{eq:dTetadt}
			\frac{\partial T_\eta}{\partial t} = \frac{1}{2}\tau[R_\eta (\lambda B-C_t)^*B]^2+\frac{1}{2}\tau[R_\eta B (\lambda B-C_t)] + \varepsilon \tau[R_\eta B^2]^2.
		\end{equation}
		In other words, $T_\eta$ satisfies the PDE
		\begin{equation}
			\label{eq:TetaPDE}
			\frac{\partial T_\eta}{\partial t} = \frac{1}{2}\left(\frac{\partial T_\eta}{\partial\lambda}\right)^2+\frac{1}{2}\left(\frac{\partial T_\eta}{\partial\bar\lambda}\right)^2 + \varepsilon \left(\frac{\partial T_\eta}{\partial\varepsilon}\right)^2.
		\end{equation}
	\end{proposition}
	\begin{proof}
		By ~\eqref{eq:AB}, we have $dC_t = i\,d\sigma_t\,B$. Recall from~\eqref{eq:AB} that $B$ is self-adjoint. Therefore, by It\^o product rule \eqref{eq;Itoprod}, we have
		\[dm_t = -i(\lambda B - C_t)^*\,d\sigma_t\,B+iB\,d\sigma_t\,(\lambda B-C_t)+ B^2dt.\]
		We compute
		\[\frac{\tau[R_\eta\,dm_t)]}{dt} = \tau[R_\eta B^2]\]
		and
		\begin{align*}
			\frac{\tau[R_\eta\,dm_t\,R_\eta\,dm_t]}{dt}& = -\tau[R_\eta (\lambda B-C_t)^*B]^2-\tau[R_\eta B (\lambda B-C_t)]^2 \\
			&\qquad + 2\tau[R_\eta B^2]\tau[R_\eta(\lambda B - C_t)^*(\lambda B - C_t)].
		\end{align*}
		The factor $\tau[R_\eta(\lambda B - C_t)^*(\lambda B - C_t)]$ be written as $1-\varepsilon\tau[R_\eta B^2]$. By Proposition~\ref{ito.prop},
		\[\frac{\partial T_\eta}{\partial t} = \frac{1}{2}\tau[R_\eta (\lambda B-C_t)^*B]^2+\frac{1}{2}\tau[R_\eta B (\lambda B-C_t)]^2 + \varepsilon \tau[R_\eta B^2]^2\]
		which is~\eqref{eq:dTetadt}. The PDE~\eqref{eq:TetaPDE} then follows from Lemma~\ref{lem:Tpartials}.
	\end{proof}
	\subsection{Proof of Theorem~\ref{thm:UnboundedPDE}}
	\label{sect:PDEproof}
	\begin{proof}
		By \eqref{eq:TetaPDE}, 
		\begin{equation}
			\label{eq:Tetaintegral}
			T_\eta(t,\lambda,\varepsilon) = \mathrm{Re}\left[\int_0^t \left(\left(\frac{\partial T_\eta}{\partial\lambda}\right)^2+ \varepsilon \left(\frac{\partial T_\eta}{\partial\varepsilon}\right)^2\right)ds\right].
		\end{equation}
		We have, using Lemma~\ref{lem:Tpartials}, the inequality
		\[\frac{\partial T_\eta}{\partial\varepsilon}\leq \frac{1}{\varepsilon}\]
		for all $\eta$. By Lemma~\ref{lem:Tpartials},
		\begin{align*}\frac{\partial T_\eta}{\partial\lambda} &= \tau[R_\eta(\lambda B - C_t)^*B]\\
			&=\tau[(\lambda-x_0-i\sigma_t)((\lambda-x_0-i\sigma_t)^*(\lambda-x_0-i\sigma_t)+B^{-2}\eta+\varepsilon)^{-1}]
		\end{align*}
		so we can apply Cauchy-Schwarz inequality to get
		\[\left\vert\frac{\partial T_\eta}{\partial\lambda}\right\vert^2\leq \tau[\vert\lambda-x_0-i\sigma_t\vert^2(\vert\lambda-x_0-i\sigma_t\vert^2+B^{-2}\eta+\varepsilon)^{-2}]\leq \frac{1}{4\varepsilon}.\]
		By Proposition~\ref{prop:partials}, the partial derivatives of $S$ with respect to $\lambda$, $\bar\lambda$, and $\varepsilon$ are pointwise limit of the corresponding partial derivatives of $T_\eta$. Since $S$ is the pointwise limit of $T_\eta+\tau[\log(\vert x_0\vert^2+1)]$ by \eqref{eq:Sptwise}, by applying the dominated convergence theorem to \eqref{eq:Tetaintegral},
		\[S(t,\lambda,\varepsilon)-\tau[\log(\vert x_0\vert^2+1)] = \mathrm{Re}\left[\int_0^t \left(\left(\frac{\partial S}{\partial\lambda}\right)^2+ \varepsilon \left(\frac{\partial S}{\partial\varepsilon}\right)^2\right)ds\right]\]
		By Lemma~\ref{lem:lambdaCont}, since $\sigma_t$ has the same distribution as $\sqrt{t}\sigma$, the integrand is continuous with respect to the parameter $s$. Differentiating the above equation with respect to $t$ gives the PDE of $S$. The initial condition of $S$ follows from the fact that $x_0$ is self-adjoint.
	\end{proof}

\section{Hamilton--Jacobi analysis}
\label{sect:HJ}
\subsection{Solving the PDE}
As in \cite{HallHo2020}, we use the Hamilton--Jacobi method to analyze the function $S$ through the PDE in Theorem~\ref{thm:UnboundedPDE}. We define the Hamiltonian function $H:\R^6\to\R$ by replacing the partial derivatives of $S$ in the PDE by the corresponding momentum variables and putting an overall minus sign. That is, 
\[H(u,v,\varepsilon, p_u, p_v, p_\varepsilon)=-\frac{1}{4}(p_u^2-p_v^2)-\varepsilon p_\varepsilon^2.\]
Then, we introduce the Hamiltonian system (the Hamilton's equations) for $H$:
\begin{equation}
	\label{eq:ODE}
\begin{array}{lll}
	\displaystyle\frac{du}{dt} = \frac{\partial H}{\partial p_u}; & \displaystyle\frac{dv}{dt} = \frac{\partial H}{\partial p_v}; & \displaystyle\frac{d\varepsilon}{dt} = \frac{\partial H}{\partial p_\varepsilon};\\[10pt]
		\displaystyle\frac{dp_u}{dt} = -\frac{\partial H}{\partial u}; & \displaystyle\frac{dp_v}{dt} = -\frac{\partial H}{\partial v}; & \displaystyle\frac{dp_\varepsilon}{dt} = -\frac{\partial H}{\partial \varepsilon}.
	\end{array}
\end{equation}
We note that the Hamiltonian function and the system of ODEs are the same as the one in \cite{HallHo2020} because the PDE is the same.

Write $\lambda=u+iv$. The initial conditions of $u$, $v$, $\varepsilon$ for the ODEs~\eqref{eq:ODE} can be chosen arbitrarily; these initial conditions are denoted by $u_0$, $v_0$, $\varepsilon_0$. The initial momenta $p_u$, $p_v$, $p_\varepsilon$ are denoted by $p_{u,0}$, $p_{v,0}$, and $p_0$ respectively, and are chosen depending on $u_0, v_0, \varepsilon_0$ as
\begin{equation}
	\label{eq:momenta}
	\begin{split}
	p_{u,0} &= \frac{\partial}{\partial u_0}S(0, \lambda_0, \varepsilon_0)=\int\frac{2(u_0-x)\,d\mu_{x_0}(x)}{(u_0-x)^2+v_0^2+\varepsilon_0};\\
	p_{v,0} &= \frac{\partial}{\partial v_0}S(0, \lambda_0, \varepsilon_0)=\int\frac{2v_0\,d\mu_{x_0}(x)}{(u_0-x)^2+v_0^2+\varepsilon_0};\\
	p_{0} &= \frac{\partial}{\partial \varepsilon_0}S(0, \lambda_0, \varepsilon_0)=\int\frac{d\mu_{x_0}(x)}{(u_0-x)^2+v_0^2+\varepsilon_0}.
	\end{split}
	\end{equation}
We have written $\lambda_0=u_0+iv_0$. Note that $S$ is defined with $\varepsilon>0$. The $\varepsilon_0$ that we choose here is also positive.

Using the Hamilton--Jacobi method, the solution $S$ of the PDE can be expressed in terms of the initial conditions, as stated in the following proposition. For the proof for the proposition, readers are referred to \cite[Proposition 4.2]{HallHo2020}. For a more general statements, see \cite[Section 6.1]{DHK2019} and \cite{EvansBook}.
\begin{proposition}
	\label{prop:HJformulas}
	Suppose that the solution to the Hamiltonian system~\eqref{eq:ODE} exists on a time interval $[0,T)$ such that $\varepsilon(t)>0$ for all $t\in[0,T)$. Then we have, for all $t\in[0,T)$,
	\[S(t,\lambda(t),\varepsilon(t)) = S(0,\lambda_0,\varepsilon_0)+tH_0\]
	where $H_0=H(u_0,v_0,\varepsilon_0,p_{u,0}, p_{v,0}, p_0)$. Also, we have
	\[\frac{\partial S}{\partial y}(t,\lambda(t),\varepsilon(t)) = p_{y}(t)\]
	for all $y\in\{u, v, \varepsilon\}$ and $t\in[0,T)$.
	\end{proposition}

\subsection{Solving the ODEs}
The PDE in our case (Theorem~\ref{thm:UnboundedPDE}) is the same as the one in \cite{HallHo2020}. In this section, we briefly summarize the results from \cite{HallHo2020} about the solutions of the ODEs~\eqref{eq:ODE}.

The following proposition combines \cite[Proposition 4.4]{HallHo2020} and \cite[Definition 4.6]{HallHo2020}.
\begin{proposition}
	\label{prop:ODEsol}
	Suppose that the initial conditions $\lambda_0=u_0+iv_0$, $\varepsilon_0$ are chosen, and initial momenta $p_{u,0}$, $p_{v,0}$, $p_0$ are determined by~\eqref{eq:momenta}. Then the solution of the Hamiltonian system~\eqref{eq:ODE} exists up to
	\begin{equation}
		\label{eq:tstar}t_\ast(\lambda_0,\varepsilon_0) = \frac{1}{p_0} = \left(\int\frac{d\mu_{x_0}(x)}{(u_0-x)^2+v_0^2+\varepsilon_0}\right)^{-1}.
		\end{equation}
	Up until $t_\ast(\lambda_0,\varepsilon_0)$, we have
	\[
	\begin{array}{ll}
	\displaystyle u(t) = u_0-\frac{t}{2}p_{u,0}; &\displaystyle  p_u(t) = p_{u,0};\\[10pt]
	\displaystyle v(t) = v_0+\frac{t}{2}p_{v,0}; &\displaystyle  p_v(t) = p_{v,0};\\[10pt]
	\displaystyle \varepsilon(t) = \varepsilon_0(1-tp_0)^2; & \displaystyle p_{\varepsilon}(t) = \frac{p_0}{1-tp_0}.
	\end{array}
		\]
	\end{proposition}

For a fixed $\lambda_0\in\C$, $t_\ast(\lambda_0,\varepsilon_0)$ defined in \eqref{eq:tstar} is strictly decreasing in $\varepsilon_0$. For each $\lambda_0\in\C$, the lifetime of the solution path in the limit $\varepsilon_0\to0^+$ is given by
\begin{equation}
\label{eq:T}
T(\lambda_0) = \lim_{\varepsilon_0\to 0^+}  t_\ast (\lambda_0,\varepsilon_0) = \left(\int\frac{d\mu_{x_0}(x)}{(u_0-x)^2+v_0^2}\right)^{-1}.
\end{equation}
By Proposition~\ref{prop:ODEsol}, informally there are two ways to make $\varepsilon(t) = 0$. The first one is to take $\varepsilon_0 = 0$; the second one is to choose $\varepsilon_0$ such that $p_0=1/t$. The first scheme works only if the lifetime of the solution path remains greater than $t$ in the limit $\varepsilon_0\to 0^+$; that is, if $T(\lambda_0)\geq t$. If the first scheme does not work (that is, if $T(\lambda_0)< t$), one needs to use the second scheme to achieve $\varepsilon(t) = 0$. 

We now identify the set of points $\lambda_0\in\C$ such that $T(\lambda_0)<t$. This set of points plays a crucial role in \cite{HallHo2020,HoZhong2019}. For proof, see \cite[Proposition 5.2]{HallHo2020}.
\begin{proposition}
	\label{prop:LambdatId}
The set 
\begin{equation}
	\label{eq:Lambdat}
	\Lambda_{x_0,t} =\{\lambda_0\in\C\left\vert T(\lambda_0)<t\right.\}
	\end{equation}
	can be identified as 
	\[\Lambda_{x_0,t}=\left.\{u_0+iv_0\in\C\right\vert\left\vert v_0\right\vert<v_{x_0,t}(u_0)\}\]
	where $v_{x_0,t}$ is defined in Definition~\ref{def:SCConv}(2). In particular, we can identify $\Lambda_{x_0,t}\cap\R  = \left.\{u_0\in\R\right\vert v_{x_0,t}(u_0)>0\}$. Furthermore, we have $\mathrm{supp}\,\mu_{x_0}\subset\overline{\Lambda_{x_0,t}}$.
\end{proposition}
For notational convenience, we also define
\begin{equation}
	\label{eq:e0p1}
	\begin{split}
	\varepsilon_0^t(u_0+iv_0) &= v_{x_0,t}(u_0)^2-v_0^2\\
	p_1 &= \int\frac{x\,d\mu_{x_0}(x)}{(u_0-x)^2+v_0^2+\varepsilon_0}.
	\end{split}
	\end{equation}
Note that $\varepsilon_0^t(\lambda_0)>0$ for $\lambda_0\in\Lambda_{x_0,t}$ and $\varepsilon_0^t(\lambda_0)=0$ for $\lambda_0\in\partial\Lambda_{x_0,t}$. Then we have the following solution at time $t$. 
\begin{proposition}
	\label{prop:limitase}
	For initial conditions $\lambda_0\in\C$ and $\varepsilon_0>0$, denote by $\lambda(t;\lambda_0,\varepsilon_0)$ and $\varepsilon(t;\lambda_0,\varepsilon_0)$ the solution to system~\eqref{eq:ODE} at time $t$.
	\begin{enumerate}
		\item If $\lambda_0\in\Lambda_{x_0,t}^c$, then $t_\ast(\lambda_0,\varepsilon_0)>t$  for all $\varepsilon_0>0$, and
		\begin{align*}
			\lim_{\varepsilon_0\to 0^+}\lambda(t;\lambda_0,\varepsilon_0) &= \lambda_0-t\int\frac{d\mu_{x_0}(x)}{\lambda_0-x}\\
			\lim_{\varepsilon_0\to 0^+}\varepsilon(t;\lambda_0,\varepsilon_0)&=0.
			\end{align*}
		\item If $\lambda_0\in\Lambda_{x_0,t}$, then $\varepsilon_0^t(\lambda_0)>0$ and $t_\ast(\lambda_0,\varepsilon_0^t(\lambda_0))=t$. Furthermore, we have
		\begin{align*}
			\lim_{\varepsilon_0\to \varepsilon_0^t(\lambda_0)^+} \lambda(t;\lambda_0,\varepsilon_0)& = tp_1+2iv_0\\
			\lim_{\varepsilon_0\to\varepsilon_0^t(\lambda_0)^+} \varepsilon(t;\lambda_0,\varepsilon_0) &= 0.
			\end{align*}
		\end{enumerate}
	\end{proposition}
\begin{proof}
	Points 1 and 2 are Propositions 4.7 and 4.5 in \cite{HallHo2020} respectively. The choice $\varepsilon_0^t(\lambda_0)$ can be found in Section 7.3 of \cite{HallHo2020}, right before Lemma 7.7.
	\end{proof}

\section{The Brown measure of $x_0+i\sigma_t$\label{sect:BrownSc}}
\subsection{The domain}
Define the holomorphic map
\begin{equation}
\label{eq:H-t}
H_{x_0,-t}(z) = z - t G_{x_0}(\lambda_0),\quad\lambda_0\not\in\mathrm{supp}\,\mu_{x_0}
\end{equation}
and
\[\Omega_{x_0,t} = [H_{x_0,-t}(\Lambda_t^c)]^c.\]
We note that the notation $H_{x_0,-t}$ is consistent with the one in Definition~\ref{def:SCConv}(1). We will show in the later section that, as in \cite{HallHo2020}, the Brown measure of $x_0+i\sigma_t$ has mass $1$ on the open set $\Omega_{x_0,t}$. 

Most of the statements in the following theorem follow the same proof in \cite[Proposition 5.5]{HallHo2020}. We only supply the proof for the statements that requires a new proof.
\begin{theorem}
	\label{thm:domain}
	The following statements hold.
	\begin{enumerate}
		\item The function $H_{x_0,-t}$ is continuous and injective on $\overline{\Gamma}_t$. (See Definition~\ref{def:SCConv} for the definition of $\Gamma_t$.)
		\item Define the function $f_{x_0,t}:\R\to\R$ by
		\begin{equation}
			\label{eq:at}
			f_{x_0,t}(u_0) = \mathrm{Re}[H_{x_0,-t}(u_0+iv_{x_0,t}(u_0))].
			\end{equation}
		Then at any point $u_0$ such that $v_{x_0,t}(u_0)>0$, the function $f_{x_0,t}$ is differentiable and
		\[0<\frac{df_{x_0,t}}{du_0}<2.\]
		\item The function $f_{x_0,t}$ is continuous and strictly increasing. It maps $\R$ onto $\R$. In particular, the inverse $f_{x_0,t}^{-1}$ of $f_{x_0,t}$ exists.
		\item Define the function
		\begin{equation}
		\label{eq:varphiDef}
		\varphi_{x_0,t}(u) = 2v_{x_0,t}(f_{x_0,t}^{-1}(u)),\quad u\in\R.
		\end{equation}
		Then $\varphi_{x_0,t}(u)\to 0$ as $|u|\to\infty$ and the map $H_{x_0,-t}$ maps the graph of $v_{x_0,t}$ to the graph of $\varphi_{x_0,t}$.
		\item The map $H_{x_0,-t}$ takes the region above the graph of $v_{x_0,t}$ onto the region above the graph of $\varphi_{x_0,t}$.
		\item The set $\Omega_{x_0,t}$ can be identified as
		\[\Omega_{x_0,t} = \left.\{u+iv\in\C\right\vert \left\vert v\right\vert< \varphi_{x_0,t}(u)\}.\]
		\end{enumerate}
	\end{theorem}
	\begin{proof}
		Points 1 and 2 follow the same proof as in \cite[Proposition 5.5]{HallHo2020}. The argument in \cite[Proposition 5.5]{HallHo2020} also shows that $f_{x_0,t}$ is continuous and strictly increasing. We now prove that $f_{x_0,t}$ maps $\R$ \emph{onto} $\R$. We can write $f_{x_0,t}$ as
		\[f_{x_0,t}(u_0) = u_0 - t\int\frac{u_0-x}{(u_0-x)^2+v_t(u_0)^2}d\mu(x).\]
		By the Cauchy--Schwarz inequality,
		\begin{align*}
			&\left\vert\int\frac{u_0-x}{(u_0-x)^2+v_t(u_0)^2}d\mu(x)\right\vert^2\\
			&\leq \int\frac{(u_0-x)^2}{(u_0-x)^2+v_t(u_0)^2}d\mu(x)\int\frac{d\mu(x)}{(u_0-x)^2+v_t(u_0)^2}\leq \frac{1}{t}.
		\end{align*}
		It follows that $\lim_{u_0\to-\infty}f_{x_0,t}(u_0) = -\infty$ and $\lim_{u_0\to\infty}f_{x_0,t}(u_0)=\infty$. Since $f_{x_0,t}$ is increasing, it maps $\R$ onto $\R$. 
		
		For Point 4, it is clear that $H_{x_0,-t}$ maps the graph of $v_{x_0,t}$ to the graph of a function $\varphi_{x_0,t}$ defined on the range of $f_{x_0,t}$ using the same formula in~\eqref{eq:varphiDef}. To show that $\varphi_{x_0,t}(u)\to 0$ as $\vert u\vert\to\infty$, it suffices to show that $v_{x_0,t}(u_0)\to 0$ as $|u_0|\to\infty$ since $\vert f_{x_0,t}^{-1}(u)\vert\to\infty$ as $\vert u\vert\to\infty$. Suppose, on the contrary, that there exist $\delta>0$ and a sequence $u_n\in\R$, such that $|u_n|\to\infty$ but $v_{x_0,t}(u_n)>\delta$ for all $n$. For all of these $u_n$, since $v_{x_0,t}(u_n)>\delta>0$, we must have
		\[\int\frac{d\mu_{x_0}(x)}{(u_n-x)^2+\delta^2}\geq \int\frac{d\mu_{x_0}(x)}{(u_n-x)^2+v_{x_0,t}(u_n)^2}=\frac{1}{t}.\]
		Letting $n\to\infty$, by dominated convergence theorem, we have $0\geq 1/t>0$, which is a contradiction. Thus, $v_{x_0,t}(u_0)\to 0$ as $|u_0|\to\infty$.
		 
		Finally, Points 5 and 6 follow the same proof as in \cite[Proposition 5.5]{HallHo2020}.
		\end{proof}
	\subsection{Surjectivity}
	We can write $f_{x_0,t}$ into the form
	\begin{equation}
		\label{eq:fx0t}
		f_{x_0,t}(u_0)=\mathrm{Re}[H_{x_0,-t}(u_0+iv_{x_0,t}(u_0)] = t\int\frac{x\,d\mu_{x_0}(x)}{(u_0-x)^2+v_{x_0,t}(u_0)^2}.
		\end{equation}
	Denote by $\lambda(t;\lambda_0,\varepsilon_0)$ and $\varepsilon(t;\lambda_0,\varepsilon_0)$ be the solution of~\eqref{eq:ODE}, with initial conditions $\lambda_0, \varepsilon_0$. By~\eqref{eq:fx0t} and Proposition~\ref{prop:limitase}, for any $\lambda_0=u_0+iv_0\in\Lambda_{x_0,t}$, 
	\begin{equation}
		\label{eq:surjMap}
		\lim_{\varepsilon_0\to \varepsilon_0^t(\lambda_0)^+} \lambda(t;\lambda_0,\varepsilon_0)= f_{x_0,t}(u_0)+2iv_0,
		\end{equation}
	which is in $\Omega_{x_0,t}$ since $2|v_0|< \varphi_{x_0,t}(f_{x_0,t}(u_0))$, by Point 4 of Theorem~\ref{thm:domain}. The following theorem shows that every point in $\Omega_{x_0,t}$ can be obtained from $\Lambda_{x_0,t}$ through the formula~\eqref{eq:surjMap}. The proof in \cite[Theorem 7.3]{HallHo2020} still works for our case.
	
	\begin{theorem}
		\label{thm:surj}
		Let 
		\[U_t(u_0+iv_0) = f_{x_0,t}(u_0)+2iv_0, \quad u_0+iv_0\in\Lambda_{x_0,t}\]
		which is the same as the formula~\eqref{eq:surjMap}. Then we have the following results.
		\begin{enumerate}
			\item The map $U_t$ extends continuously to $\overline{\Lambda}_{x_0,t}$. This extension is the unique continuous map of $\overline{\Lambda}_{x_0,t}$ into $\overline{\Omega}_{x_0,t}$ that agrees with $H_{x_0,-t}$ on $\partial\Lambda_{x_0,t}$ and maps each vertical segment in $\overline{\Lambda}_{x_0,t}$ linearly to a vertical segment in $\overline{\Omega}_{x_0,t}$.
			\item The map $U_t$ is a homeomorphism from $\Lambda_{x_0,t}$ onto $\Omega_{x_0,t}$.
			\end{enumerate} 
		\end{theorem}
	
The following proposition gives bounds on the real parts of points in $\Omega_{x_0,t}$. 
	\begin{proposition}
		Let $M = \sup\mathrm{supp}(\mu)$ and $m =\inf\mathrm{supp}(\mu)$. Then every point $\lambda\in\overline{\Omega}_{x_0,t}$ satisfies
		\[m<\mathrm{Re}\lambda<M.\]
		\end{proposition}
	\begin{proof}
		The proof of \cite[Proposition 7.4]{HallHo2020} works for this proposition. The difference from \cite[Proposition 7.4]{HallHo2020} is that $M$ and $m$ may not be both finite --- could be one infinite and one finite, or both infinite.
		\end{proof}

\subsection{The Brown measure computation}
In this section, we compute the Brown measure of $x_0+i\sigma_t$ by taking the Laplacian of $S$ in $\Omega_{x_0,t}$. By the definition of Brown measure (See Section~\ref{sect:BrownDef}), the Laplacian is in distributional sense. The following regularity result tells us that we can indeed take the ordinary Laplacian to compute the density of the Brown measure. The proof can be found in \cite[Proposition 7.5]{HallHo2020}
\begin{proposition}
	\label{prop:regularity}
	Define the function $\tilde{S}$ by
	\begin{equation}
		\label{eq:Stilde}
		\tilde{S}(t,\lambda,z) = S(t,\lambda,z^2)
		\end{equation}
	for $z>0$. Then for any $t>0$ and $\lambda^*\in\Omega_{x_0,t}$, the function
	\[(\lambda,z)\mapsto \tilde{S}(t,\lambda,z)\]
	extends to a real analytic function in a neighborhood of $(\lambda^*,0)$ in $\C\times \R$.
	\end{proposition}
Let 
\[s_t(\lambda) = \lim_{\varepsilon\to 0^+}S(t,\lambda,\varepsilon) = \lim_{\varepsilon\to 0^+}\tilde{S}(t,\lambda,\sqrt{\varepsilon})\]
where $\tilde{S}$ is defined in~\eqref{eq:Stilde}. The computation of the Laplacian $s_t$ is the same as in \cite{HallHo2020}. Since the computation of the Laplacian is crucial for the computation of the Brown measure, for completeness, we also include the computation of the Laplacian here.

Write $\lambda(t;\lambda_0,\varepsilon_0)$ be the solution of the system~\eqref{eq:ODE} with initial conditions $\lambda_0$ and $\varepsilon_0$.  By~\eqref{eq:surjMap} and Theorem~\ref{thm:surj}, for any $\lambda=u+iv\in\Lambda_{x_0,t}$, $\lambda_0 =u_0+iv_0 = U_t^{-1}(\lambda)$ satisfies \[\lim_{\varepsilon_0\to \lambda_0^t(\lambda_0)^+}\lambda(t;\lambda_0,\varepsilon_0) = \lambda.\] 
Then, by Propositions~\ref{prop:HJformulas} and~\ref{prop:regularity}, for any $\lambda=u+iv\in\Omega_{x_0,t}$,
\[\frac{\partial}{\partial u}s_t(\lambda) = \lim_{\varepsilon\to 0^+} \frac{\partial}{\partial u}\tilde{S}(t,\lambda, \varepsilon)=\lim_{\varepsilon_0\to \varepsilon_0^t(\lambda_0)^+}p_{u}(t) = \lim_{\varepsilon_0\to\varepsilon_0^t(\lambda_0)^+}p_{u,0}.\]
Using~\eqref{eq:momenta}, \eqref{eq:fx0t}, and Proposition~\ref{prop:limitase}(2).
\begin{align*}
	&\lim_{\varepsilon_0\to\varepsilon_0^t(\lambda_0)^+}p_{u,0} = \int\frac{2(u_0-x)\,d\mu_{x_0}(x)}{(u_0-x)^2+v_{x_0,t}(u_0)^2}\\
	&=\frac{2u_0}{t}-\frac{2}{t}\cdot t\int\frac{x\,d\mu_{x_0}(x)}{(u_0-x)^2+v_{x_0,t}(u_0)^2}\\
	& = \frac{2f_{x_0,t}^{-1}(u)}{t}-\frac{2u}{t}.
\end{align*}
Hence, we have
\begin{equation}
	\label{eq:uDer}
	\frac{\partial s_t}{\partial u} = \frac{2f_{x_0,t}^{-1}(u)}{t}-\frac{2u}{t}, \quad \lambda=u+iv\in\Omega_{x_0,t}.
	\end{equation}
Similarly, since $v = 2v_0$ by~\eqref{eq:surjMap}, for all $\lambda=u+iv\in\Omega_{x_0,t}$,
\begin{equation}
	\label{eq:vDer}
	\frac{\partial}{\partial v}s_t(\lambda) = \lim_{\varepsilon_0\to\varepsilon_0^t(\lambda_0)^+}p_{v,0} = \int\frac{2v_0\,d\mu_{x_0}(x)}{(u_0-x)^2+v_{x_0,t}(u_0)^2} = \frac{v}{t}.
\end{equation}
Combining \eqref{eq:uDer} and \eqref{eq:vDer}, we have
\[\Delta s_t(\lambda) =\frac{2}{t}\frac{d}{du}f_{x_0,t}^{-1}(u)-\frac{1}{t}=\frac{2}{t}\left(\frac{d}{du}f_{x_0,t}^{-1}(u)-\frac{1}{2}\right).\]

We arrive at the Brown measure of $x_0+i\sigma_t$ as well as a push-forward result of the Brown measure. 
\begin{theorem}
	\label{thm:a+ist}
	The open set $\Omega_{x_0,t}$ is a set of full measure for the Brown measure of $x_0+i\sigma_t$. The Brown measure is absolutely continuous with a strictly positive density $w_t$ on the open set $\Omega_{x_0,t}$ by
	\[w_t(\lambda) = \frac{1}{2\pi t}\left(\frac{d}{du}f_{x_0,t}^{-1}(u)-\frac{1}{2}\right),
	\quad \lambda=u+iv\in\Omega_{x_0,t}.\]
	In particular, the density $w_t$ is constant along vertical segments inside $\Omega_{x_0,t}$. Moreover, the push-forward of the Brown measure of $x_0+i\sigma_t$ by
	\[Q_t(u+iv) = 2f_{x_0,t}^{-1}(u)-u,\quad u+iv\in\Omega_{x_0,t}\]
	is the law of $x_0+\sigma_t$.
	\end{theorem}
\begin{proof}
	These results follow from the same proof as for Theorems 7.9 and 8.2 of \cite{HallHo2020} since Biane's Theorem (Theorem~\ref{thm:Biane}) on the free convolution with a semicircular variable holds for unbounded self-adjoint random variable $x_0\in\A^\Delta$.
	\end{proof}

Theorem 8.2 of \cite{HallHo2020} also establishes a push-forward property of the Brown measure of $x_0+i\sigma_t$ to the Brown measure of $x_0+c_t$ where $c_t$ is the circular variable with variance $t$. We have not computed the Brown measure of $x_0+c_t$ for unbounded $x_0$. We will establish this analogous push-forward result to $x_0+c_t$ in the next section.

\section{The Brown measure of $x_0+c_{\alpha,\beta}$\label{sect:ellipse}}
Recall that an elliptic variable has the form
\[c_{\alpha,\beta} = \tilde{\sigma}_\alpha+i\sigma_\beta\]
where $\tilde{\sigma}_\alpha$ and $\sigma_\beta$ are free semicircular random variables with variances $\alpha$ and $\beta$ respectively. Given any self-adjoint $a\in\tilde{\A}$, let
\[\Lambda_{a,s} = \{\left.u+iv\in\C\right\vert |v|<v_{a,s}(u)\},\]
where $v_{a,s}$ is defined in Definition~\ref{def:SCConv}(2). We note that this definition of $\Lambda_{a,s}$ is consistent with the notation in the previous section by Proposition~\ref{prop:LambdatId}.

In this section, we study the Brown measure of $x_0+c_{\alpha,\beta}$, where $x_0\in\A^\Delta$ is self-adjoint, and $\alpha\geq 0$ (the case $\alpha=0$ agrees with the one computed in Theorem~\ref{thm:a+ist}). The computation in \cite{Ho2020} relies only on the computations of holomorphic functions; the compactness of the support of $\mu_{x_0}$ in \cite{Ho2020} does not play a role (since the computation of the holomorphic maps can be restricted to a Stolz angle). The Brown measure computation of \cite{Ho2020} can be automatically carried over to the unbounded self-adjoint $x_0$.

We start by a few notations; these notations are well-defined for $a\in\tilde{\A}$. Later we will apply these notations to $x_0\in\A^\Delta$ since the Brown measure of $x_0+c_{\alpha,\beta}$ is defined only for $x_0\in\A^\Delta$ but not for $x_0\in\tilde{\A}$. More discussions can be found in Proposition 3.6 of~\cite{Ho2020}.
\begin{definition}\label{def:ForEllipse}
	\begin{enumerate}
		\item Given any self-adjoint $a\in\tilde{\A}$ and $r\in\R$. Define 
		\[H_{a,r}(z) = z + r G_a(z),\quad z\not\in\sigma(a).\]
		This definition is consistent with~\eqref{eq:Htdef}, when $r>0$, and consistent with~\eqref{eq:H-t}, when $r<0$. Furthermore, $H_{a,r}$ is an injective conformal map on $\Lambda_{a,|r|}^c$.
		\item Let $a\in\tilde{\A}$ be self-adjoint. Also let $\alpha\geq 0$ and $\beta>0$ and write $s=\alpha+\beta$. Define $f_{a,\alpha,\beta}:\R\to\R$ by
		\[f_{a,\alpha,\beta}(u) = \mathrm{Re}[H_{\alpha-\beta}(u+iv_{x_0,s}(u))].\]

		This function is strictly increasing and is a homeomorphism onto $\R$. Furthermore, $f_{a,\alpha,\beta}'(u)>0$ for all $u\in\Lambda_{a,s}\cap\R$.
		\item Let $\alpha> 0$ and $\beta>0$ and write $s=\alpha+\beta$. Define \[\varphi_{x_0,\alpha,\beta}(u) = 2\beta v_{a,s}(f_{x_0,\alpha,\beta}^{-1}(u))/s\] and 
		\[\Omega_{x_0,\alpha,\beta} = \left.\{u+iv\in\C\right\vert |v|<\varphi_{x_0,\alpha,\beta}(u)\}.\]
		By Definition 3.1 and Proposition 3.6 of \cite{Ho2020}, $\Omega_{x_0,\alpha,\beta}=[H_{x_0,\alpha-\beta}(\Lambda_{x_0,s}^c)]^c$. 	Remark that $\varphi_{x_0,0,t} = \varphi_{x_0,t}$, which is defined in Theorem~\ref{thm:domain}(4).
	\end{enumerate}
	\end{definition}

	We will see that $\Omega_{x_0,\alpha,\beta}$ is an open set of full measure with respect to the Brown measure of $x_0+c_{\alpha,\beta}$.
	\begin{remark}
		The author uses a different parametrization for the elliptic variable in \cite{Ho2020}. In \cite{Ho2020} the elliptic variable is parametrized using $s$ and $t$ by $\tilde{\sigma}_{s-t/2}+i\sigma_{t/2}$. The parametrization in \cite{Ho2020} is motivated by \cite{Ho2016}. In this paper, the parametrization is a linear transformation of the one in \cite{Ho2020}, by $\alpha = s-t/2$ and $\beta = t/2$.
	\end{remark}
	
	If we apply Theorem~\ref{thm:a+ist} to compute the Brown measure of $x_0+c_{\alpha,\beta}=x_0+\tilde{\sigma}_\alpha+i\sigma_\beta$, the Brown measure is written in terms of the law of $x_0+\tilde{\sigma}_\alpha$. The key to write the Brown measure in terms of $x_0$ is the following proposition in \cite{Ho2020}. 
	\begin{proposition}[Theorem 3.4 of \cite{Ho2020}]
		\label{prop:Forellipse}
		Let $a = x_0+\tilde{\sigma}_\alpha$. Then, by writing $f_{x_0,\alpha,\beta}(u_0)=u$, we have
		\[f_{a,\beta}^{-1}(u) = u+\beta\int\frac{(u_0-x)\,d\mu_{x_0}(x)}{(u_0-x)^2+v_{x_0,s}(u_0)^2},\quad u\in\R.\]
	\end{proposition}
	We are ready to state the main result in this case.
	\begin{theorem}
		\label{thm:ellipse}
		The Brown measure of $x_0+c_{\alpha,\beta}$ is absolutely continuous with respect to the Lebesgue measure on the plane. The open set $\Omega_{\alpha,\beta}$ is a set of full measure of the Brown measure. The density of the Brown measure on $\Omega_{\alpha,\beta}$ is given by
		\[w_{\alpha,\beta}(u+iv) = \frac{1}{4\pi \beta}\left(1+2\beta\frac{d}{du}\int\frac{(u_0-x)\,d\mu_{x_0}(x)}{(u_0-x)^2+v_{x_0,s}(u_0)^2}\right),\quad u+iv\in\Omega_{\alpha,\beta},\]
		where $f_{x_0,\alpha,\beta}(u_0)=u$. The density is strictly positive and constant along the vertical segments in $\Omega_{x_0,\alpha,\beta}$.
	\end{theorem}
	\begin{remark}
		\label{rem:EllipseToSc}
		If $\alpha = 0$ and $\beta = t$, then Proposition~\ref{prop:Forellipse} implies that
		\[w_{\alpha,\beta}(u+iv) = \frac{1}{4\pi t}\left(1+2\frac{d}{du}(f_{x_0,t}^{-1}(u)-u)\right)=\frac{1}{2\pi t}\left(\frac{d}{du}f_{x_0,\beta}^{-1}(u)-\frac{1}{2}\right),\]
		giving the density of the Brown measure of $x_0+i\sigma_t$.
		\end{remark}

	When $x_0\in\A$ is a bounded self-adjoint random variable, the Brown measure of $x_0+c_t$ is computed in \cite{HoZhong2019}. The following corollary computes the Brown measure of $x_0+c_t$ for possibly-unbounded self-adjoint $x_0\in\A^\Delta$. The highlights are, as in \cite{HoZhong2019}, the Brown measure has full measure on the open set $\Lambda_{x_0,t}$, and the density can be written as terms of the derivative of the function $\psi_{x_0,t}$ in Theorem~\ref{thm:Biane}(3).
	
	\begin{corollary}
		\label{cor:Circular}
		The Brown measure of $x_0+c_t$ is absolutely continuous with respect to the Lebesgue measure on the plane. The open set $\Lambda_{x_0,t}$ is a set of full measure of the Brown measure. The density on $\Lambda_{x_0,t}$ has the form
		\begin{align*}
			w_{t/2,t/2}(u+iv)& =\frac{1}{2\pi t}\frac{d\psi_{x_0,t}}{du}\\
			&= \frac{1}{\pi t}\left(1-\frac{t}{2}\frac{d}{du}\int\frac{x\,d\mu_{x_0}(x)}{(u-x)^2+v_{x_0,t}(u)^2}\right),\quad u+iv\in\Lambda_{x_0,t},
		\end{align*}
		where $\psi_{x_0,t}$ is defined in Theorem~\ref{thm:Biane}. The density is strictly positive and constant along the vertical segments in $\Lambda_{x_0,t}$.
		\end{corollary}
	\begin{proof}
		Take $\alpha = \beta =t/2$ in Theorem~\ref{thm:ellipse}. Then, $\alpha+\beta=t$ and $2\beta/(\alpha+\beta)=1$, so that $f_{x_0,\alpha,\beta}$ is the identity map, and
		\[\Omega_{x_0,t/2,t/2} = \left.\{u+iv\in\C\right\vert |v|<v_{x_0,t}(u)\} = \Lambda_{x_0,t}.\]
		Then we can compute the density using the formula in Theorem~\ref{thm:ellipse}
		\begin{align*}
		w_{t/2,t/2}(u+iv) &= \frac{1}{2\pi t}\left(1+t\frac{d}{du}\int\frac{(u-x)\,d\mu_{x_0}(x)}{(u-x)^2+v_{x_0,t}(u)^2}\right)\\
		&=\frac{1}{2\pi t}\frac{d\psi_{x_0,t}}{du},\quad u+iv\in\Lambda_{x_0,t};
		\end{align*}
		the last equality follows from the definition of $\psi_{x_0,t}(u) = H_{x_0,t}(u+iv_{x_0,t}(u))$. By Definiton~\ref{def:SCConv}(2) of $v_{x_0,t}$, we have
		\[\int\frac{(u-x)\,d\mu_{x_0}(x)}{(u-x)^2+v_{x_0,t}(u)^2} = \frac{u}{t}-\int\frac{x\,d\mu_{x_0}(x)}{(u-x)^2+v_{x_0,t}(u)^2}.\]
		The other claimed formula of $w_t$ follows from a straightforward algebraic computation.
		\end{proof}
	The following corollary establishes the push-forward properties between the Brown measure of $x_0+c_{\alpha,\beta}$, the Brown measure of $x_0+c_s$, and the law of $x_0+\sigma_t$, where $s=\alpha+\beta$ is the total variance of the elliptic variable $c_{\alpha,\beta}$.
	\begin{corollary}
		\label{cor:ellipsePush}
		Write $s=\alpha+\beta$.	Let $U_{\alpha,\beta}:\Lambda_{x_0,s}\to\Omega_{x_0,\alpha,\beta}$ be defined by
		\begin{equation}
		\label{eq:ellipticU}
		U_{\alpha,\beta}(u+iv) = f_{x_0,\alpha,\beta}(u)+i\frac{2\beta v}{s}.
		\end{equation}
		Then $U_{\alpha,\beta}$ extends to a homeomorphism from $\overline{\Lambda}_{x_0,s}$ to $\overline{\Omega}_{\alpha,\beta}$ and agrees with $H_{x_0,\alpha-\beta}$ on the boundary of $\Lambda_{x_0,s}$. Furthermore, the following push-forward properties hold.
		\begin{enumerate}
			\item The push-forward of the Brown measure of $x_0+c_s$ under the map $U_{\alpha,\beta}$ is the Brown measure of $x_0+c_{\alpha,\beta}$.
			\item The push-forward of the Brown measure of $x_0+c_{\alpha,\beta}$ by the map
			\[Q_{\alpha,\beta}(u+iv) = \begin{cases}
			\frac{1}{\alpha-\beta}[su-2\beta f_{x_0,\alpha,\beta}^{-1}(u)],\quad &\textrm{if $\alpha\neq \beta$}\\
			H_{x_0,s}(u+iv_{x_0,s}(u)),\quad &\textrm{if $\alpha= \beta=s/2$}.
			\end{cases}\]
			is the law of $x_0+\sigma_s$.
			\end{enumerate}
		\end{corollary}
	\begin{remark}
		Write $u = f_{x_0,\alpha,\beta}(u_0)$. When $\alpha\neq \beta$, we can compute that
		\[Q_{\alpha,\beta}(u+iv) =u_0+s\int\frac{(u_0-x)\,d\mu_{x_0}(x)}{(u_0-x)^2+v_{x_0,s}(u_0)^2}=H_{x_0,s}(u_0+iv_{x_0,s}(u)).\]
		If $\alpha=\beta=s/2$, then $u=u_0$ and the above equation reduces to $H_{x_0,s}(u+iv_{x_0,s}(u))$, which is the definition of $Q_{\alpha,\beta}$ when $\alpha=\beta=s/2$. 
		\end{remark}
	\begin{proof}
		That $U_{\alpha,\beta}$ extends to a homeomorphism from $\overline{\Lambda}_{x_0,s}$ to $\overline{\Omega_{\alpha,\beta}}$ and agrees with $H_{x_0,\alpha-\beta}$ on the boundary of $\Lambda_{x_0,s}$ follows from the proof of Proposition 4.2 of \cite{Ho2020}.
		
		For $\alpha\neq\beta$, the proof follows from the one given for Theorem 4.1 of \cite{Ho2020}. For the case $\alpha=\beta=s/2$, the proof follows from the one given for Theorem 3.13 of \cite{HoZhong2019}.
		\end{proof}
	
Using the push-forward property in Point 1 of Corollary~\ref{cor:ellipsePush}, we can express the density of the Brown measure of $x_0+c_{\alpha,\beta}$ in terms of the density of the Brown measure of $x_0+c_{s}$, if $\alpha+\beta=s$. 
\begin{corollary}
	For any $r=2\beta/s$, by writing $u+iv=U_{\alpha,\beta}(u_0+iv_0)$ for all $u_0+iv_0\in\Lambda_{x_0,s}$, we have
	\[w_{\alpha,\beta}(u+iv) = \frac{1}{r}\frac{w_{s/2,s/2}(u_0+iv_0)}{r+2\pi(1-r)s\cdot w_{s/2,s/2}(u_0+iv_0)}.\]
	\end{corollary}
\begin{proof}
	This follows from a direct computation of integration of substitution using Corollary~\ref{cor:ellipsePush}. Interested readers are referred to the proof of Corollary 4.3 of \cite{Ho2020}.
	\end{proof}

\begin{figure}[h]
	\begin{center}
		\includegraphics[width=0.65\textwidth]{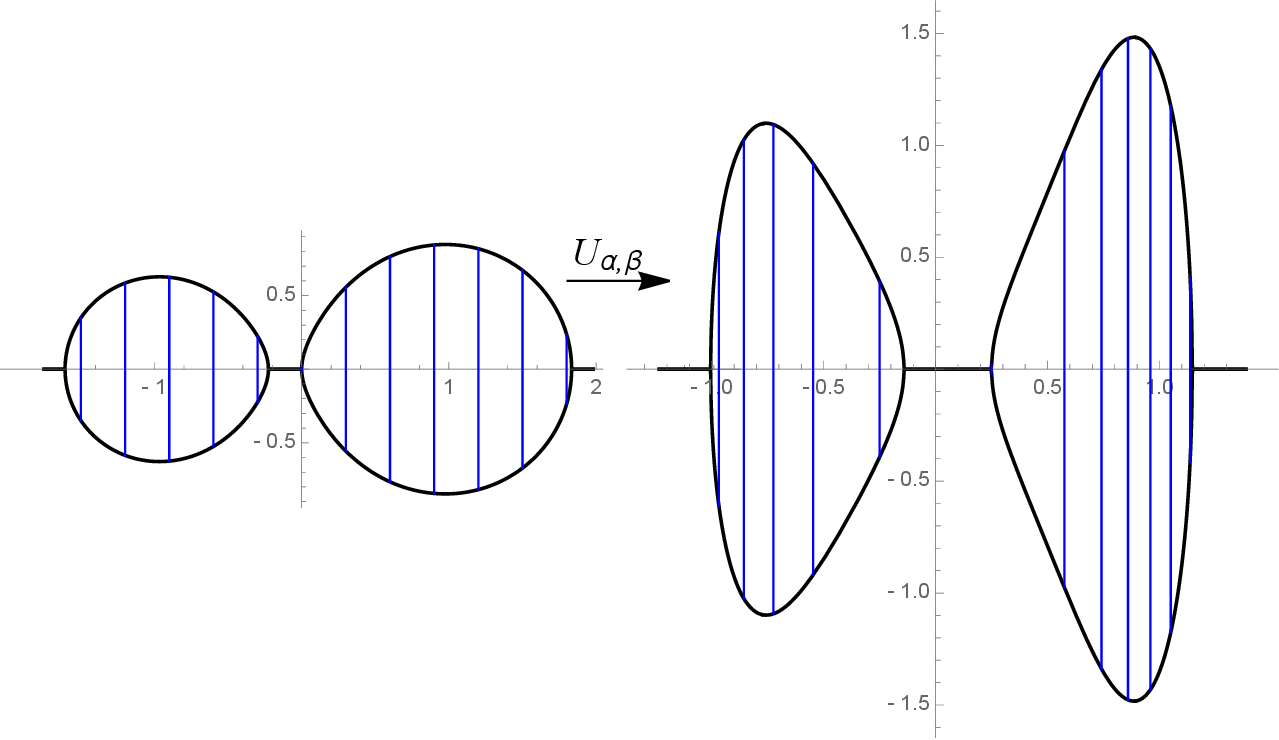}
	\end{center}
	\caption{Visualization of the map $U_{\alpha,\beta}$ with $\alpha = 1/8$ and $\beta =7/8$.\label{fig:UVisual}}
\end{figure}

Figure~\ref{fig:UVisual} shows a visualization of the map $U_{\alpha,\beta}$, with $\mu_{x_0} = \frac{1}{3}\delta_{-1}+\frac{2}{3}\delta_1$, $\alpha = 1/8$ and $\beta = 7/8$. The map $U_{\alpha,\beta}$ takes the blue equally-spaced vertical lines on the left hand side of the figure to the corresponding vertical lines on the right hand side of the figure. The blue vertical lines on the right hand side are no longer equally-spaced. In the next proposition, we investigate the spacing of the vertical lines on the right hand side by looking at the second derivative of $f_{x_0,\alpha,\beta}$. If $f_{x_0,\alpha,\beta}''>0$ on an interval $I$, then the spacing between the image of the vertical lines in $\Lambda_{x_0,s}$ is increasing on $I$. Since the push-forward of the Brown measure of $x_0+c_s$ by $U_{\alpha,\beta}$ is the Brown measure of $x_0+c_{\alpha,\beta}$, this proposition describes how mass is transferred under this push-forward map --- whether the mass is ``squeezed'' or ``stretched'' (relative to the total mass on $\Omega_{x_0,\alpha,\beta}$). More precisely, the mass of
\[\{\left.u_0+iv_0\in\Lambda_{x_0,s}\right\vert a\leq u_0\leq b\}\]
with respect to the Brown measure $x_0+c_s$ is transferred to the set
\[\{\left.u+iv\in\Omega_{x_0,\alpha,\beta}\right\vert f_{x_0,\alpha,\beta}(a)\leq u\leq f_{x_0,\alpha,\beta}(b)\}.\]
If $f_{x_0,\alpha,\beta}''<0$ on $(a,b)$ and $f_{x_0,\alpha,\beta}''>0$ on $(b,c)$, then the push-forward map $U_{\alpha,\beta}$ ``squeezes'' the mass towards the vertical line intersecting $b$. Similarly, if $f_{x_0,\alpha,\beta}''>0$ on $(a,b)$ and $f_{x_0,\alpha,\beta}''<0$ on $(b,c)$, then the push-forward map $U_{\alpha,\beta}$ ``stretches'' the mass away from the vertical line intersecting $b$. 

\begin{proposition}
	\label{prop:spacing}
	Writing $s=\alpha+\beta$, we have
	\[f_{x_0,\alpha,\beta}''(u_0) =2\pi(\alpha-\beta)w_{s/2,s/2}'(u_0)= \frac{\alpha-\beta}{s}\psi_{x_0,s}''(u_0)\]
	for all $u_0\in\Lambda_{x_0,s}\cap\R$, where $\psi_{x_0,s}(u_0) = H_{x_0,s}(u_0+iv_{x_0,s}(u_0))$ is defined in Theorem~\ref{thm:Biane}(3).
\end{proposition}
\begin{proof}
	Using the formula
	\[f_{x_0,\alpha,\beta}(u_0) = u_0+(\alpha-\beta)\int\frac{(u_0-x)\,d\mu_{x_0}(x)}{(u_0-x)^2+v_{x_0,s}(u_0)^2}=u_0+\frac{\alpha-\beta}{s}(\psi_{x_0,s}(u_0)-u_0),\]
	we have
	\[f_{x_0,\alpha,\beta}''(u_0) = \frac{\alpha-\beta}{s}\psi_{x_0,s}''(u_0).\]
	The other formula for $f_{x_0,\alpha,\beta}''$ follows from $\psi_{x_0,s}'(u_0) = 2\pi s w_{s/2,s/2}(u_0)$ by Corollary~\ref{cor:Circular}.
	\end{proof}

\section{Examples: Cauchy case\label{sect:example}}
In this section, we compute the Brown measures of $x_0+c_{\alpha,\beta}$ where $x_0$ has the Cauchy distribution
\[d\mu_{x_0}(x) = \frac{1}{\pi}\frac{dx}{1+x^2}.\]
Since the density of $\mu_{x_0}$ has polynomial decay at $\pm\infty$, $x_0\in\A^\Delta$, so does $x_0+c_{\alpha,\beta}\in\A^\Delta$. We first compute the Brown measure of $x_0+c_t$; the computation in the process is also useful for the computation of the Brown measure of $x_0+c_{\alpha,\beta}$. Lastly, we compute the Brown measure of $x_0+i\sigma_t$ by putting $\alpha = 0$ and $\beta = t$ to the Brown measure of $x_0+c_{\alpha,\beta}$.
\subsection{Adding a circular variable}
In this section, we compute the Brown measure of $x_0+c_t$ as in the following theorem. Some of the computations will be used again when we compute the Brown measure of $x_0+c_{\alpha,\beta}$.
\begin{theorem}
	\label{thm:CauchyEx}
	When $x_0$ is Cauchy distributed, the boundary of the domain $\Lambda_{x_0,t}$ has the form
	\begin{equation}
	\label{eq:BdryLambdatCauchy}
	\partial\Lambda_{x_0,t} = \left\{u\pm iv\left\vert u^2=\frac{1}{v}(1+v)(t-v-v^2), v>0\right.\right\}.
	\end{equation}
	The upper boundary $(\partial\Lambda_{x_0,t})\cap(\C^+\cup\R)$  of $\Lambda_{x_0,t}$ is the graph of a positive unimodal function with peak $\frac{-1+\sqrt{1+4t}}{2}$ at $0$.
	The Brown measure of $x_0+c_t$ has full measure on $\Lambda_{x_0,t}$, with density
	\[w_{t/2,t/2}(\lambda) = \frac{1}{2\pi t}\frac{t+4v^2(1+v)^2}{(1+v)(t+2v^2(1+v))}\]
	where $v = v_{x_0,t}(u)$.
\end{theorem}
\begin{remark}
	As $|u|\to\infty$, the density $w_t(u)$ does not approach $0$. In fact, as $|u|\to\infty$, $v_{x_0,t}(u)\to 0$, so $w_t(u)$ approaches $1/(2\pi t)$. The density $w_t(\lambda)$ on $\Lambda_{x_0,t}$ still defines a probability measure because the function $v_{x_0,t}(u)\approx t/u^2$ as $|u|\to\infty$.
	\end{remark}

Figure~\ref{fig:CauchyCircular} plots the eigenvalue simulation of $x_0+c_t$, the density of the Brown measure of $x_0+c_t$, as well as the function $w_t(u)$ for $u\in\R$ at $t=1$.

\begin{figure}[h]
	\begin{center}
	\includegraphics[width=0.6\textwidth]{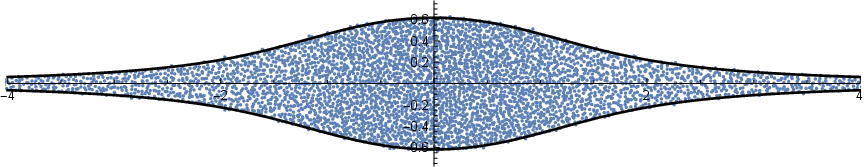}\\[10pt]
	\includegraphics[width=0.6\textwidth]{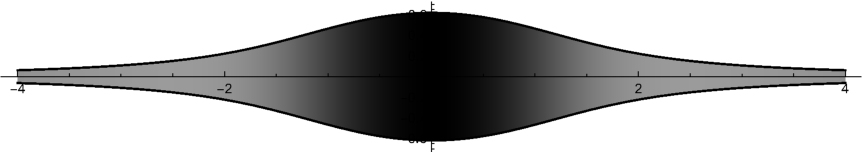}\\[10pt]
	\includegraphics[width=0.6\textwidth]{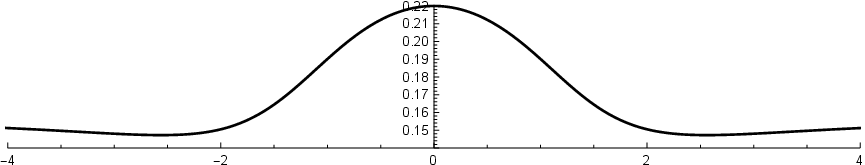}
	\caption{$5000\times 5000$ matrix simulation of eigenvalues of $x_0+c_t$ (top), density plot of the Brown measure of $x_0+c_t$ (middle), and $w_{t/2,t/2}(u)$ for $u\in\R$ (bottom), at $t=1$\label{fig:CauchyCircular}. Top two plotted with boundary of $\Lambda_{x_0,t}$.}
	\end{center}
	\end{figure}
We start by computing the function $v_{x_0,t}(u)$, then we compute the derivative of $\psi_{x_0,t}$ (See Theorem~\ref{thm:Biane} for definition of $\psi_{x_0,t}$).
\begin{lemma}
	For each $u\in\R$, $v_{x_0,t}(u)$ is the unique positive number $v$ satisfying
	\begin{equation}
	\label{eq:vEqCauchy}
	u^2=\frac{1}{v}(1+v)(t-v-v^2).
	\end{equation}
	Thus, $u$ and $dv_{x_0,t}/du$ have opposite sign; in particular, $v_{x_0,t}$ is unimodal with peak at $0$, and $v_{x_0,t}(0)=\frac{-1+\sqrt{1+4t}}{2}$.
	\end{lemma}
\begin{proof}
	Let $u\in\R$. We can compute (computer software such as Mathematica could be helpful) that
	\[\frac{1}{t} = \int\frac{d\mu_{x_0}(x)}{(u-x)^2+v_{x_0,t}(u)^2} = \frac{1+v_{x_0,t}(u)}{v_{x_0,t}(u)(u^2+(1+v_{x_0,t}(u))^2)},\]
	so that
	\[u^2=\frac{1}{v_{x_0,t}(u)}[1+v_{x_0,t}(u)][t-v_{x_0,t}(u)-v_{x_0,t}(u)^2].\]
	This proves \eqref{eq:vEqCauchy}. 
	
	Differentiating both sides of \eqref{eq:vEqCauchy} with respect to $u$, we have
	\begin{equation}
	\label{eq:vDerivative}
	2u = \left(-\frac{t}{v_{x_0,t}(u)^2}-2(1+v_{x_0,t}(u))\right)\frac{dv_{x_0,t}}{du}
	\end{equation}
	which shows $u$ and $dv_{x_0,t}/du$ have oppposite sign since we must have $-t/v_{x_0,t}(u)-2(1+v_{x_0,t}(u))<0$. It follows that $dv_{x_0,t}/du<0$ for $u>0$ and $dv_{x_0,t}/du>0$ for $u<0$, proving $v_{x_0,t}$ is unimodal with peak at $0$. The value $v_{x_0,t}(0)$ comes from solving the equation~\eqref{eq:vEqCauchy} at $u=0$.
	\end{proof}

\begin{proposition}
	\label{eq:Cauchypsi}
	We have
	\[\frac{d\psi_{x_0,t}}{du}=\frac{t+4v^2(1+v)^2}{(1+v)(t+2v^2(1+v))},\]
	where $v = v_{x_0,t}(u)$.
	\end{proposition}
\begin{proof}
	Since
	\[\frac{1+v}{v(u^2+(1+v)^2)}=\int\frac{d\mu_{x_0}(x)}{(u-x)^2+v_{x_0,t}(u)^2}=\frac{1}{t},\]
	we can compute (again, computer software could be useful)
\begin{equation}
\label{eq:psiCauchy}
\psi_{x_0,t}(u) = u+t\int\frac{(u-x)\,d\mu_{x_0}(x)}{(u-x)^2+v_{x_0,t}(u)^2} = u+t\frac{u}{u^2+(1+v)^2} = u+\frac{uv}{1+v}.
\end{equation}
Using~\eqref{eq:vEqCauchy} and~\eqref{eq:vDerivative}, we can compute the derivative
\begin{equation}
\label{eq:ReDer}
\frac{d}{du}\frac{uv}{1+v} = \frac{v}{1+v}+\frac{u}{(1+v)^2}\frac{dv}{du}=\frac{v(-t+2v(1+v)^2)}{(1+v)(t+2v^2(1+v))}
\end{equation}
and so
\[\frac{d\psi_{x_0,t}}{du}=\frac{t+4v^2(1+v)^2}{(1+v)(t+2v^2(1+v))},\]
completing the proof.
\end{proof}

\begin{proof}[Proof of Theorem~\ref{thm:CauchyEx}]
	By Proposition~\ref{prop:LambdatId},$\Lambda_{x_0,t} = \{\left.u+iv\in\C\right\vert |v|<v_{x_0,t}(u)\}$; thus, \eqref{eq:BdryLambdatCauchy} and the description of the boundary of $\Lambda_{x_0,t}$ follow from~\eqref{eq:vEqCauchy}. 
	
	The formula of the density of the Brown measure follows from Corollary~\ref{cor:Circular} and Proposition~\ref{eq:Cauchypsi}. 
	\end{proof}
\subsection{Adding an elliptic variable\label{sect:Cauchyelliptic}}
The main result about the Brown measure of $x_0+c_{\alpha,\beta}$ is as follows.
\begin{theorem}
	\label{thm:CauchyEllipse}
	When $x_0$ is Cauchy distributed, the function $\varphi_{x_0,\alpha,\beta}$ in Definition~\ref{def:ForEllipse}(3) is unimodal with peak $\frac{\beta(-1+\sqrt{1+4s})}{s}$ at $u = 0$. The boundary of the corresponding $\Omega_{x_0,\alpha,\beta}$ (See Definition~\ref{def:ForEllipse}(3)) has the form
	\begin{equation}
	\label{eq:CauchyEllBdry}
	\partial \Omega_{x_0,\alpha,\beta}=\left\{u\pm ib\in\C\left\vert u^2 = \frac{(b\alpha+\beta)^2(4\beta^2-2b\beta-b^2s)}{b\beta^2(bs+2\beta)}\right.\right\}.
	\end{equation}
	The function $b=\varphi_{x_0,\alpha,\beta}(u)$ has a decay of order $1/u^2$ as $|u|\to\infty$.
	
	The density of the Brown measure has the form
	\[w_{\alpha,\beta}(\lambda) = \frac{1}{4\pi \beta}\frac{b^4 s^3+4b^3s^2\beta+4b^2 s\beta^2+4\beta^4}{b^4 s^2\alpha+4b^3s\alpha\beta+4b^2\alpha\beta^2+4b\beta^4+4\beta^4},\quad \lambda\in\Omega_{x_0,\alpha,\beta}\]
	where $b = \varphi_{x_0,\alpha,\beta}(u)$.
	\end{theorem}
\begin{remark}
	The density $w_{\alpha,\beta}$ does not approach $0$ as $u=\mathrm{Re}(\lambda)$ approaches infinity. In fact, as $|u|\to\infty$, $\varphi_{x_0,\alpha,\beta}(u)\to 0$ and so
	\[\lim_{|u|\to\infty}w_{\alpha,\beta}(u) = \frac{1}{4\pi \beta}.\]
	\end{remark}
Figure~\ref{fig:CauchyElliptic} plots an eigenvalue simulation of $x_0+c_{\alpha,\beta}$, the density of the Brown measure of $x_0+c_{\alpha,\beta}$, and the graph of the function $w_{\alpha,\beta}(u)$ for $u\in\R$ at $\alpha=1/8$ and $\beta=7/8$.
\begin{figure}[h]
	\begin{center}
		\includegraphics[width=0.6\textwidth]{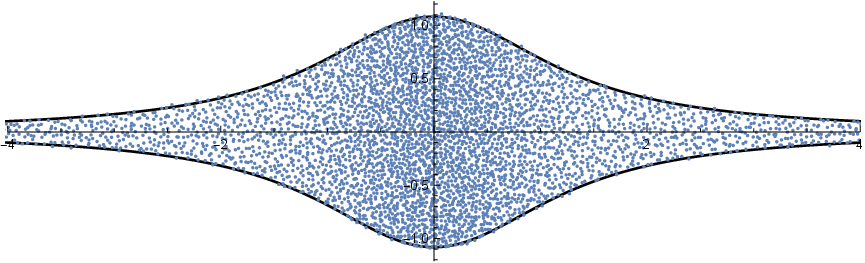}\\[10pt]
		\includegraphics[width=0.6\textwidth]{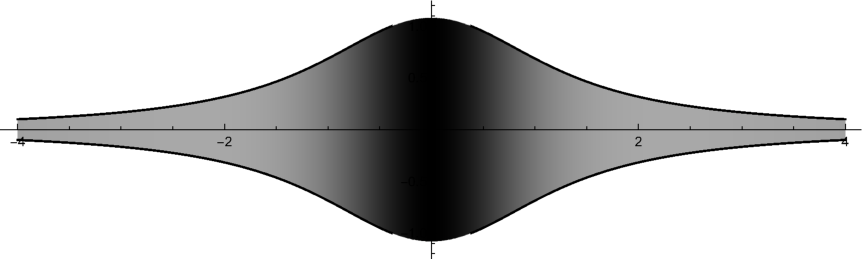}\\[10pt]
		\includegraphics[width=0.6\textwidth]{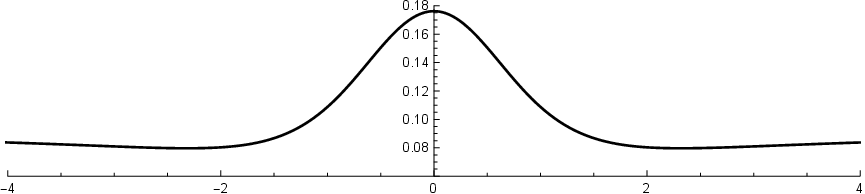}
		\caption{$5000\times 5000$ matrix simulation of eigenvalues of $x_0+c_{\alpha,\beta}$ (top), density plot of the Brown measure of $x_0+c_{\alpha,\beta}$ (middle), and $w_{\alpha,\beta}(u)$ for $u\in\R$ (bottom), at $\alpha=1/8$ and $\beta=7/8$. Top two plotted with boundary of $\Omega_{x_0,\alpha,\beta}$.\label{fig:CauchyElliptic}}
	\end{center}
\end{figure}

We start proving Theorem~\ref{thm:CauchyEllipse} by the following lemma which concerns the derivative of $f_{x_0,\alpha,\beta}$.
\begin{lemma}
	\label{lem:dudu0}
	Let $u = f_{x_0,\alpha,\beta}(u_0)$ for $u_0\in\R$. Then, by writing $s=\alpha+\beta$ and $v=v_{x_0,s}(u_0)$,
	\[\frac{du}{du_0} = \frac{(1+v)(s^2+4\alpha v^2(1+v))-(\alpha-\beta)sv}{s(1+v)(s+2v^2(1+v))}.\]
	\end{lemma}
\begin{proof}
	Using a computation similar to~\eqref{eq:psiCauchy}, we have
	\begin{equation}
	\label{eq:uu0}
	\begin{split}
	u=f_{x_0,\alpha,\beta}(u_0) &= u_0+(\alpha-\beta)\int\frac{(u_0-x)\,d\mu_{x_0}(x)}{(u_0-x)^2+v^2}\\
	&=u_0+\frac{\alpha-\beta}{s}\frac{u_0v}{1+v}.
	\end{split}
	\end{equation}
	 Then using~\eqref{eq:ReDer} with $s$ in place of $t$ and $u_0$ in place of $u$, we have
	\[\frac{du}{du_0} = 1+\frac{\alpha-\beta}{s}\frac{v(-s+2v(1+v)^2)}{(1+v)(s+2v^2(1+v))}.\]
	The conclusion now follows from an algebraic computation.
	\end{proof}

\begin{lemma}
	\label{lem:dduRe}
	Let $u = f_{x_0,\alpha,\beta}(u_0)$ for $u_0\in\R$. Then, by writing $s=\alpha+\beta$ and $v=v_{x_0,s}(u_0)$, we have
	\[\frac{d}{du}\int\frac{(u_0-x)\,d\mu_{x_0}(x)}{(u_0-x)^2+v} = \frac{v(-s+2v(1+v)^2)}{(1+v)(s^2+4\alpha v^2(1+v))-(\alpha-\beta)sv}.\]
	\end{lemma}
\begin{proof}
	Using a computation similar to~\eqref{eq:psiCauchy}, we have
	\[\int\frac{(u_0-x)\,d\mu_{x_0}(x)}{(u_0-x)^2+v} = \frac{1}{s}\frac{u_0v}{1+v}.\]
	Thus, by a computation similar to~\eqref{eq:ReDer} and Lemma~\ref{lem:dudu0},
	\begin{align*}
	\frac{d}{du}\int\frac{(u_0-x)\,d\mu_{x_0}(x)}{(u_0-x)^2+v}& = \frac{1}{s}\frac{v(-s+2v(1+v)^2)}{(1+v)(s+2v^2(1+v))}\frac{du_0}{du}\\
	&= \frac{v(-s+2v(1+v)^2)}{(1+v)(s^2+4\alpha v^2(1+v))-(\alpha-\beta)sv},
	\end{align*}
	completing the proof.
	\end{proof}
\begin{proof}[Proof of Theorem~\ref{thm:CauchyEllipse}]
	First, the boundary of $\Omega_{x_0,\alpha,\beta}$ is the image of the graph 
	\[\{(u_0, v_{x_0,t}(u_0))\vert\,u_0\in\R\}\]
	 under the map $U_{\alpha,\beta}$ (See~\eqref{eq:ellipticU}). By~\eqref{eq:vEqCauchy} and~\eqref{eq:uu0}, if we write $u=f_{x_0,\alpha\beta}(u_0)$, then
	\[u^2 = \frac{(1+v)(s-v-v^2)}{v}\left(1+\frac{\alpha-\beta}{s}\frac{v}{1+v}\right)^2=\frac{1}{s^2}\frac{s-v-v^2}{v}\frac{(s+2\alpha v)^2}{1+v}\]
	where $v = v_{x_0,s}(u_0)$. The claimed formula of $\partial\Omega_{x_0,\alpha,\beta}$ follows from the fact that $b = \varphi_{x_0,\alpha,\beta}(u) = 2\beta v/s$. It is clear that $\varphi_{x_0,\alpha,\beta}(u)$ has a decay of order $1/u^2$ as $|u|\to\infty$, because $v\to0$ as $|u|\to\infty$.
	
	Recall from Theorem~\ref{thm:CauchyEx} that $v_{x_0,s}(0) = \frac{-1+\sqrt{1+4s}}{2}$.  Since $f_{x_0,\alpha,\beta}(0) = 0$ by~\eqref{eq:uu0} and $f_{x_0,\alpha,\beta}$ is strictly increasing, the function
	\[\varphi_{x_0,\alpha,\beta}(u) = \frac{2\beta v_{x_0,s}(u_0)}{s}\]
	is unimodal with peak $\frac{\beta(-1+\sqrt{1+4s})}{s}$ at $0$.
	
	By Theorem~\ref{thm:ellipse} and Lemma~\ref{lem:dduRe}, we have
		\begin{align*}
		w_{\alpha,\beta}(u+iv) &= \frac{1}{4\pi \beta}\left(1+2\beta\frac{d}{du}\int\frac{(u_0-x)\,d\mu_{x_0}(x)}{(u_0-x)^2+v_{x_0,s}(u_0)^2}\right)\\
		&= \frac{1}{4\pi\beta}\left(1+2\beta\frac{v(-s+2v(1+v)^2)}{(1+v)(s^2+4\alpha v^2(1+v))-(\alpha-\beta)sv}\right)\\
		&= \frac{1}{4\pi \beta}\frac{b^4 s^3+4b^3s^2\beta+4b^2 s\beta^2+4\beta^4}{b^4 s^2\alpha+4b^3s\alpha\beta+4b^2\alpha\beta^2+4b\beta^4+4\beta^4}
		\end{align*}
	where we have used $b = \varphi_{x_0,\alpha,\beta}(u) = (2\beta/s)v$.
	\end{proof}

We close this section by an example illustrating how mass is transformed by the push-forward under $U_{\alpha,\beta}$ (See~\eqref{eq:ellipticU}), as proved in Proposition~\ref{prop:spacing}. By Proposition~\ref{prop:spacing} and Theorem~\ref{thm:CauchyEx},
\[f_{x_0,\alpha,\beta}''(u_0)=2\pi(\alpha-\beta)w_{s/2,s/2}'(u_0)=(\alpha-\beta)\frac{4v(1+v)(1+3v(1+v))-s}{ s(1+v)^2(1+2v^2(1+v))^2}\frac{dv}{du_0}\]
where $v=v_{x_0,s}(u_0)$. Recall that $v_{x_0,s}(u_0)$ is unimodal, $dv/du_0=0$ if and only if $u_0=0$. As an example, when $s=1$, we can solve, by the relation of $u_0$ and $v_{x_0,s}(u_0)$ in Lemma~\ref{eq:vEqCauchy}, that $f_{x_0,\alpha,\beta}''(u_0) = 0$ if and only if $u_0 = 0$ or $u_0=\pm\frac{-3+\sqrt{15}}{6}\approx \pm 2.56141$.

The top diagram of Figure~\ref{fig:Spacing} shows vertical blue line segments inside $\Lambda_{x_0,1}$. The spacing between the blue line segments is $0.25$. The middle diagram of Figure~\ref{fig:Spacing} shows the corresponding vertical blue line segments after being mapped by $U_{1/8,7/8}$: each blue line segment intersecting $u_0\in\R$ in the top diagram is mapped to a vertical line segment intersecting $f_{x_0,1/8,7/8}(u_0)$ in the middle diagram. The bottom diagram plots the differences
\[[f_{x_0,1/8,7/8}(u_0+0.05)-f_{x_0,1/8,7/8}(u_0)]-[ f_{x_0,1/8,7/8}(u_0)-f_{x_0,1/8,7/8}(u_0-0.05)]\]
for $u_0 =0.05 k$ for $k = 0,1,\ldots,120$. Figure~\ref{fig:Spacing} agrees with the theoretical computation in the preceding paragraph. The spacings between the image of the vertical line segments with real part less than or equal $2.5$ in the top diagram are increasing, as shown in the middle diagram, whereas the spacings between the image of the vertical line segments with real part greater than $2.5$ in the top diagram are decreasing, as shown in the middle diagram. The bottom diagram also shows a sign change at a value slightly greater than $2.5$.
\begin{figure}
	\begin{center}
		\includegraphics[width=0.6\textwidth]{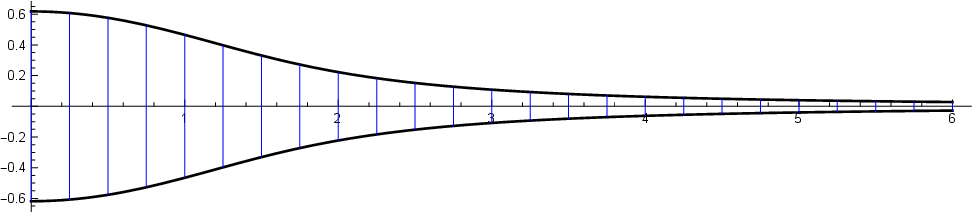} \includegraphics[width=0.6\textwidth]{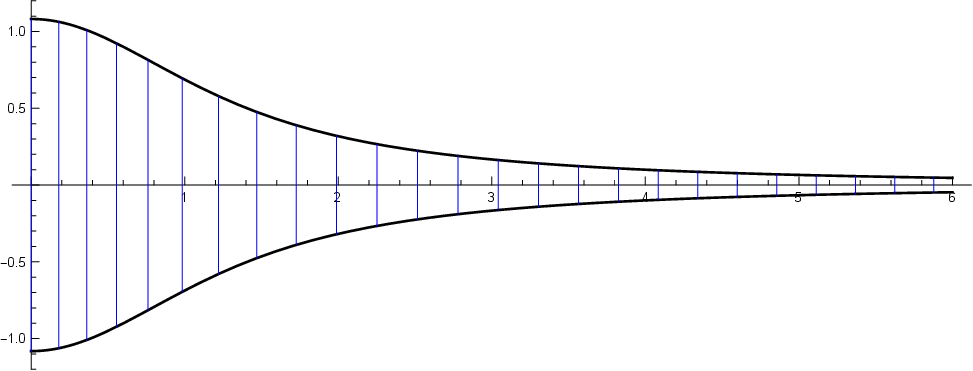}
	\includegraphics[width=0.6\textwidth]{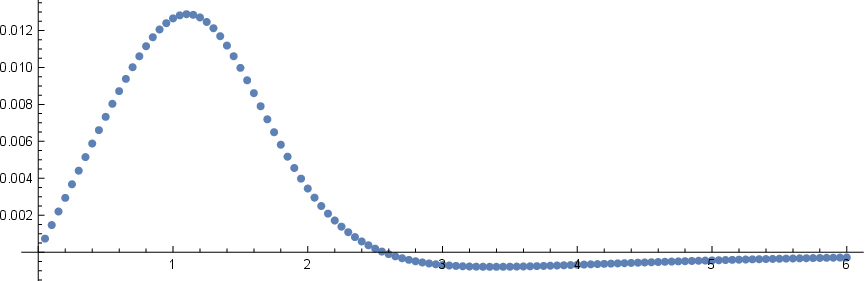}
	\caption{Equally-spaced blue lines inside $\Lambda_{x_0,1}$ (Top). The corresponding blue lines under the map $U_{1/8,7/8}$ (Middle). Differences between the lines: $f_{x_0,1/8,7/8}(u_0+0.05)-f_{x_0,7/8,7/8}(u_0)$ and $f_{x_0,1/8,7/8}(u_0)-f_{x_0,1/8,7/8}(u_0-0.05)$ (Bottom).\label{fig:Spacing}}
	\end{center}
	\end{figure}

\subsection{Adding an imaginary multiple of semicircular variable}
In this section, we take $\alpha=0$ and $\beta = t$ in Theorem~\ref{thm:CauchyEllipse} to obtain the Brown measure of $x_0+i\sigma_t$. One significant simplification over the previous cases is that the function $\varphi_{x_0,t}$, whose graph is the (upper) boundary of $\Omega_{x_0,t}$, and the density $w_t(u+iv)$ can be written explicitly in terms of $u$, instead of $v_{x_0,t}(u)$.
\begin{theorem}
	\label{thm:CauchySc}
	When $x_0$ has the Cauchy distribution, the domain $\Omega_{x_0,t}$ has the form
	\[\Omega_{x_0,t} = \left\{u+iv\in\C\left\vert |v|\leq \frac{4t}{\sqrt{u^2+1}(\sqrt{u^2+1}+\sqrt{u^2+1+4t})}\right.\right\}.\]
	The density of the Brown measure of $x_0+i\sigma_t$ is given by
	\[w_{t}(\lambda)=\frac{1}{4\pi t}\frac{4t+(1+u^2)^2}{(1+u^2)^{3/2}\sqrt{u^2+1+4t}},\quad \lambda=u+iv\in\Omega_{x_0,t}.\]
	\end{theorem}

\begin{lemma}
	\label{lem:Cauchyphi}
	When $x_0$ has the Cauchy distribution, the function $\varphi_{x_0,t}$ defined in Theorem~\ref{thm:domain}(4) can be computed as 
		\[\varphi_{x_0,t}(u) =  \frac{4t}{\sqrt{u^2+1}(\sqrt{u^2+1}+\sqrt{u^2+1+4t})},\quad u\in\R.\]
	\end{lemma}
\begin{proof}
	By putting $\alpha = 0$ and $\beta=t$ in \eqref{eq:CauchyEllBdry}, we get
	\[u^2 = \frac{t^2(4t^2-2bt-b^2t)}{bt^2(bt+2t)} =  \frac{(4t-2b-b^2)}{b(b+2)}\] 
	where $b = \varphi_{x_0,0,t}(u) = \varphi_{x_0,t}$, as indicated in Definition~\ref{def:ForEllipse}. Thus, since $\varphi_{x_0,t}(u)>0$, we can solve the above equation and get
	\[\varphi_{x_0,t}(u) = \frac{-\sqrt{u^2+1}+\sqrt{u^2+1+4t}}{\sqrt{u^2+1}} = \frac{4t}{\sqrt{u^2+1}(\sqrt{u^2+1}+\sqrt{u^2+1+4t})},\]
	completing the proof.
	\end{proof}

\begin{proof}[Proof of Theorem~\ref{thm:CauchySc}]
	The formula for $\Omega_{x_0,t}$ follows directly from Lemma~\ref{lem:Cauchyphi}. To compute the density of the Brown measure, by Remark~\ref{rem:EllipseToSc}, we again put $\alpha = 0$ and $\beta = t$ in Theorem~\ref{thm:CauchyEllipse}, and get
	\[w_{t}(\lambda) = \frac{1}{4\pi t}\frac{4t^4+tb^2(bt+2t)^2}{4t^3(bt+t)}= \frac{1}{4\pi t}\frac{4t+b^2(b+2)^2}{4t(b+1)},\quad\lambda\in\Omega_{x_0,t}.\]
	where $b = \varphi_{x_0,t}(u)$. By putting in the formula of $\varphi_{x_0,t}(u)$ in Lemma~\ref{lem:Cauchyphi}, we have
	\[w_{t}(\lambda)=\frac{1}{4\pi t}\frac{4t+(1+u^2)^2}{(1+u^2)^{3/2}\sqrt{u^2+1+4t}},\quad \lambda=u+iv\in\Omega_{x_0,t}.\]
	The theorem is established.
	\end{proof}

In the case $x_0+i\sigma_t$, the density $w_t(\lambda)$ is very explicit, in terms of $u=\mathrm{Re}(\lambda)$. It is not hard to see that
\[\lim_{|u|\to\infty}w_t(u) = \frac{1}{4\pi t}.\] 

Recall that Figure~\ref{fig:Cauchy} plots an eigenvalue simulation of $x_0+i\sigma_t$, the density of the Brown measure of $x_0+i\sigma_t$ and the function $w_{t}(u)$ for $u\in\R$ at $t=1$. Figure~\ref{fig:CauchyCompare} shows the plots of the Brown measure densities of $x_0+c_t$ ($t=1$), $x_0+c_{\alpha,\beta}$ ($\alpha = 1/8$, $\beta = 7/8$), and $x_0++i\sigma_t$ ($t=1$). It also shows the graphs of $w_{\alpha,\beta}(u)$, $u\in\R$, for comparison. Note that although we observe the trend from Figure~\ref{fig:Cauchy} that $w_{\alpha,\beta}(0)$ decreases as $\beta$ increases while keeping $\alpha+\beta = 1$, we do not lose mass around vertical strips around the imaginary axis. The map $U_{\alpha,\beta}$ indeed pushes mass towards the imaginary axis in a vertical strip including the origin, by the discussion in the last paragraph of Section~\ref{sect:Cauchyelliptic}.

\begin{figure}[h]
	\begin{center}
		\includegraphics[width=0.45\textwidth]{CauchyCircularDensity}\qquad\includegraphics[width=0.45\textwidth]{CauchyCircularW}\\[10pt]
		\includegraphics[width=0.45\textwidth]{CauchyEllipticDensity}\qquad\includegraphics[width=0.45\textwidth]{CauchyEllipticW}\\[10pt]
		\includegraphics[width=0.45\textwidth]{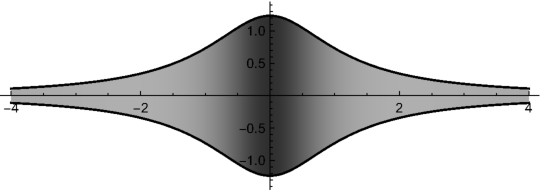}\qquad\includegraphics[width=0.45\textwidth]{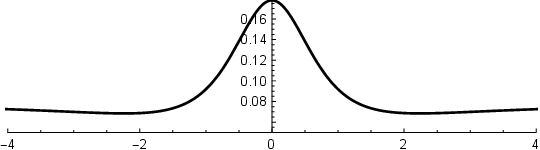}\\[10pt]
		\caption{Left: Brown measure densities of $x_0+c_1$ (top), $x_0+c_{1/8,7/8}$ (middle), and $x_0+i\sigma_1$ (bottom). Right: Graphs of $w_{1/2,1/2}$ (top), $w_{1/8,7/8}$ (middle), and $w_1$ (bottom).\label{fig:CauchyCompare}}
	\end{center}
\end{figure}

\section{Acknowledgments}
The author would like to thank Hari Bercovici and Roland Speicher who asked the author in two different seminars whether one can extend the results in \cite{HallHo2020, Ho2020} to unbounded random variables. Their question was the starting point of this paper, and the author had useful discussions with them. The author would like to express his special thank to Hari Bercovici and Brian Hall for extra discussions and reading the updated version of the manuscript which eliminates a gap in the first version of the manuscript. The author would also like to thank Marek Bo\.{z}ejko, Eugene Lytvynov for useful conversations. The author also thanks Hall for helping computer simulations and plotting graphs.

\bibliographystyle{acm}
\bibliography{UnboundedSum}

\begin{thebibliography}{10}

\bibitem{BelinschiBercovici2007}
{\sc Belinschi, S.~T., and Bercovici, H.}
\newblock A new approach to subordination results in free probability.
\newblock {\em J. Anal. Math. 101\/} (2007), 357--365.

\bibitem{BercoviciVoiculescu1993}
{\sc Bercovici, H., and Voiculescu, D.}
\newblock Free convolution of measures with unbounded support.
\newblock {\em Indiana Univ. Math. J. 42}, 3 (1993), 733--773.

\bibitem{Biane1997sc}
{\sc Biane, P.}
\newblock On the free convolution with a semi-circular distribution.
\newblock {\em Indiana Univ. Math. J. 46}, 3 (1997), 705--718.

\bibitem{Biane1998}
{\sc Biane, P.}
\newblock Processes with free increments.
\newblock {\em Math. Z. 227}, 1 (1998), 143--174.

\bibitem{BianeLehner2001}
{\sc Biane, P., and Lehner, F.}
\newblock Computation of some examples of {B}rown's spectral measure in free
  probability.
\newblock {\em Colloq. Math. 90}, 2 (2001), 181--211.

\bibitem{BianeSpeicher1998}
{\sc Biane, P., and Speicher, R.}
\newblock Stochastic calculus with respect to free {B}rownian motion and
  analysis on {W}igner space.
\newblock {\em Probab. Theory Related Fields 112}, 3 (1998), 373--409.

\bibitem{Brown1986}
{\sc Brown, L.~G.}
\newblock Lidskii's theorem in the type {II} case.
\newblock In {\em Geometric methods in operator algebras ({K}yoto, 1983)},
  vol.~123 of {\em Pitman Res. Notes Math. Ser.} Longman Sci. Tech., Harlow,
  1986, pp.~1--35.

\bibitem{DHK2019}
{\sc Driver, B.~K., Hall, B.~C., and Kemp, T.}
\newblock The {B}rown measure of the free multiplicative {B}rownian motion.
\newblock {\em arXiv:1903.11015\/} (2019).

\bibitem{EvansBook}
{\sc Evans, L.~C.}
\newblock {\em Partial {D}ifferential {E}quations}.
\newblock {V}ol 19, Graduate Studies in Mathematics. American Mathematical
  Society, Providence, RI, 1998.

\bibitem{FugledeKadison1952}
{\sc Fuglede, B., and Kadison, R.~V.}
\newblock Determinant theory in finite factors.
\newblock {\em Ann. of Math. (2) 55\/} (1952), 520--530.

\bibitem{Girko1985}
{\sc Girko, V.~L.}
\newblock The elliptic law.
\newblock {\em Teor. Veroyatnost. i Primenen. 30}, 4 (1985), 640--651.

\bibitem{HaagerupSchultz2007}
{\sc Haagerup, U., and Schultz, H.}
\newblock Brown measures of unbounded operators affiliated with a finite von
  {N}eumann algebra.
\newblock {\em Math. Scand. 100}, 2 (2007), 209--263.

\bibitem{HallHo2020}
{\sc Hall, B.~C., and Ho, C.-W.}
\newblock The {B}rown measure of the sum of self-adjoint element and an
  imaginary multiple of a semicircular element.
\newblock {\em arXiv:2006.07168\/} (2020).

\bibitem{HallHo2021}
{\sc Hall, B.~C., and Ho, C.-W.}
\newblock The {B}rown measure of a family of free multiplicative {B}rownian
  motions.
\newblock {\em arXiv:2104.07859\/} (2021).

\bibitem{Ho2016}
{\sc Ho, C.-W.}
\newblock The two-parameter free unitary {S}egal-{B}argmann transform and its
  {B}iane-{G}ross-{M}alliavin identification.
\newblock {\em J. Funct. Anal. 271}, 12 (2016), 3765--3817.

\bibitem{Ho2020}
{\sc Ho, C.-W.}
\newblock The {B}rown measure of the sum of a self-adjoint element and an
  elliptic element.
\newblock {\em arXiv:2007.06100\/} (2020).

\bibitem{HoZhong2019}
{\sc Ho, C.-W., and Zhong, P.}
\newblock {B}rown measures of free circular and multiplicative {B}rownian
  motions with self-adjoint and unitary initial conditions.
\newblock {\em arXiv:1908.08150\/} (2019).

\bibitem{KummererSpeicher1992}
{\sc K\"{u}mmerer, B., and Speicher, R.}
\newblock Stochastic integration on the {C}untz algebra {$O_\infty$}.
\newblock {\em J. Funct. Anal. 103}, 2 (1992), 372--408.

\bibitem{Maassen1992}
{\sc Maassen, H.}
\newblock Addition of freely independent random variables.
\newblock {\em J. Funct. Anal. 106}, 2 (1992), 409--438.

\bibitem{NicaSpeicherBook}
{\sc Nica, A., and Speicher, R.}
\newblock {\em Lectures on the combinatorics of free probability}, vol.~335 of
  {\em London Mathematical Society Lecture Note Series}.
\newblock Cambridge University Press, Cambridge, 2006.

\bibitem{Nikitopoulos2021}
{\sc Nikitopoulos, E.~A.}
\newblock It\^{o}'s formula for noncommutative ${C}^2$ functions of free
  it\^{o} processes with respect to circular {B}rownian motion.
\newblock {\em arXiv:2011.08493\/} (2021).

\bibitem{Sniady2002}
{\sc \'{S}niady, P.}
\newblock Random regularization of {B}rown spectral measure.
\newblock {\em J. Funct. Anal. 193}, 2 (2002), 291--313.

\bibitem{Voiculescu1986}
{\sc Voiculescu, D.}
\newblock Addition of certain non-commuting random variables.
\newblock {\em Journal of Functional Analysis 66}, 3 (1986), 323 -- 346.

\bibitem{Voiculescu1991}
{\sc Voiculescu, D.}
\newblock Limit laws for random matrices and free products.
\newblock {\em Invent. Math. 104}, 1 (1991), 201--220.

\bibitem{Voiculescu1993}
{\sc Voiculescu, D.}
\newblock The analogues of entropy and of {F}isher's information measure in
  free probability theory. {I}.
\newblock {\em Comm. Math. Phys. 155}, 1 (1993), 71--92.

\bibitem{Voiculescu2000}
{\sc Voiculescu, D.}
\newblock The coalgebra of the free difference quotient and free probability.
\newblock {\em Internat. Math. Res. Notices}, 2 (2000), 79--106.

\end{thebibliography}

\end{document}